\documentclass[12pt]{amsart}
\usepackage[latin9]{inputenc}
\usepackage{verbatim}
\usepackage{amstext}
\usepackage{amsthm}
\usepackage{amssymb}
\usepackage{graphicx}
\usepackage{amscd}
\usepackage{color}
\usepackage{etoolbox}
\usepackage[all]{xy}
\usepackage[unicode=true, bookmarks=false, breaklinks=false,pdfborder={0 0 1},backref=section,colorlinks=false]{hyperref}

\numberwithin{equation}{section}
\numberwithin{figure}{section}
\theoremstyle{plain}
\newtheorem{thm}{\protect\theoremname}[section]
  \theoremstyle{definition}
  \newtheorem{defn}[thm]{\protect\definitionname}
  \theoremstyle{definition}
  \newtheorem{example}[thm]{\protect\examplename}
  \theoremstyle{plain}
  \newtheorem{prop}[thm]{\protect\propositionname}
  \theoremstyle{remark}
  \newtheorem{rem}[thm]{\protect\remarkname}
  \theoremstyle{plain}
  \newtheorem{lem}[thm]{\protect\lemmaname}
  \theoremstyle{plain}
  \newtheorem{cor}[thm]{\protect\corollaryname}
  \theoremstyle{plain}
  \newtheorem{conjecture}[thm]{\protect\conjecturename}

\apptocmd{\sloppy}{\hbadness 10000\relax}{}{}
\apptocmd{\sloppy}{\vbadness 10000\relax}{}{}

\theoremstyle{remark}
\newtheorem*{question*}{\protect\questionname}

\theoremstyle{plain}
\newtheorem{alphatheorem}{Theorem}

\providecommand{\corollaryname}{Corollary}
\providecommand{\definitionname}{Definition}
\providecommand{\examplename}{Example}
\providecommand{\lemmaname}{Lemma}
\providecommand{\propositionname}{Proposition}
\providecommand{\remarkname}{Remark}
\providecommand{\theoremname}{Theorem}
\providecommand{\questionname}{Question}
\providecommand{\conjecturename}{Conjecture}

\def\quot{/\!\!/}
\def\Ad{\mathsf{Ad}}
\def\tr{\mathsf{tr}}

\def\hom{\mathsf{Hom}}

\def\X{\mathfrak{X}}
\def\Y{\mathfrak{Y}}
\DeclareMathOperator{\GL}{\mathrm{GL}}
\DeclareMathOperator{\SL}{\mathrm{SL}}
\DeclareMathOperator{\PSL}{\mathrm{PSL}}
\DeclareMathOperator{\SU}{\mathrm{SU}}

\def\C{\mathbb{C}}
\def\Z{\mathbb{Z}}
\def\R{\mathbb{R}}

\def\id{\mathbf{1}}
\def\KN{\mathcal{KN}}
\def\norm{|\!|}

\newcommand{\liep}{\mathfrak{p}}
\newcommand{\liepc}{\mathfrak{p}^{\mathbb{C}}}
\newcommand{\liek}{\mathfrak{k}}
\newcommand{\liekc}{\mathfrak{k}^{\mathbb{C}}}
\newcommand{\lieg}{\mathfrak{g}}
\newcommand{\liegc}{\mathfrak{g}^{\mathbb{C}}}

\DeclareMathOperator{\Fix}{Fix}

\title[Flawed groups, character varieties]{Flawed groups and the topology of character varieties.}

\author[C. Florentino]{Carlos Florentino}

\address{Dep. Matematica, Fac. Ci\^{e}ncias, Univ. Lisboa, Edf. C6, Campo Grande 1749-016 Lisboa, Portugal.} 

\email{caflorentino@fc.ul.pt}

\author[S. Lawton]{Sean Lawton}

\address{Department of Mathematical Sciences, George Mason University, 4400 
University Drive, Fairfax, Virginia 22030, USA}

\email{slawton3@gmu.edu}

\date{\today}


\makeatletter
\@namedef{subjclassname@2020}{\textup{2020} Mathematics Subject Classification}
\makeatother


\begin{document}

\begin{abstract}
A finitely presented group $\Gamma$ is called \textit{flawed} if
$\hom(\Gamma,G)\quot G$ deformation retracts onto its subspace $\hom(\Gamma,K)/K$
for all reductive affine algebraic groups $G$ and maximal compact subgroups
$K\subset G$. After discussing generalities concerning flawed groups,
we show that all finitely generated groups isomorphic to a free product
of nilpotent groups are flawed. This unifies and generalizes all previously
known classes of flawed groups. We also provide further evidence for
the authors' conjecture that RAAGs are flawed. Lastly,
we show direct products between finite groups and some flawed group
are also flawed. These latter two theorems enlarge the known class
of flawed groups. 
\end{abstract}


\maketitle
\tableofcontents{}

\section{Introduction}

Let $\Gamma$ be finitely presented group generated by $r$ elements.
Let $G$ be a connected, reductive, affine algebraic group over $\C$;
a reductive $\C$-group for short. $G$ acts by conjugation on the
set of homomorphisms $\hom(\Gamma,G)$. We may consider $\hom(\Gamma,G)$
as a subset of $G^{r}$ by identifying homomorphisms with their evaluations
at the generators of $\Gamma$. This simultaneously gives $\hom(\Gamma,G)$
the structure of an affine algebraic set, and an analytic topology
over $\C$ (which we assume throughout the paper).

Call a homomorphism in $\hom(\Gamma,G)$ \textit{polystable} if it
has a closed conjugation orbit, and let $\hom(\Gamma,G)^{*}$ be the
subspace of polystable homomorphisms. The \textit{$G$-character variety
of $\Gamma$} is the quotient space: 
\[
\X_{\Gamma}(G):=\hom(\Gamma,G)^{*}/G,
\]
with respect to the conjugation action. By \cite[Theorem 2.1]{FlLa4},
$\X_{\Gamma}(G)$ is homeomorphic (in the analytic topology) to the
Geometric Invariant Theory (GIT) quotient $\hom(\Gamma,G)\quot G$,
and so $\X_{\Gamma}(G)$ inherits a natural algebraic structure. It
is also homotopic to the non-Hausdorff quotient $\hom(\Gamma,G)/G$
by \cite[Proposition 3.4]{FLR}.

Inspired by \cite{BC}, in \cite{FlLa1} the authors show that $\X_{\Gamma}(G)$
strong deformation retracts to $\X_{\Gamma}(K)$ when $\Gamma$ is
a finitely generated free group, and $K$ is a maximal compact subgroup in
$G$. Later, in \cite{CFLO1,CFLO2} this result was extended to the
case when $G$ is \textit{real} reductive. In \cite{FlLa4} the same
theorem is shown to hold when $\Gamma$ is a finitely generated abelian
group. Generalizing the abelian case, in part by combining methods
used in \cite{FlLa1} and \cite{FlLa4}, the result is also shown
to be true for all finitely generated nilpotent groups $\Gamma$ in
\cite{Berg}. On the other hand, whenever $\Gamma$ is the fundamental
group of a closed orientable surface of genus $g\geq2$ (called a
hyperbolic surface group henceforth) such a deformation retraction
(indeed, even a homotopy equivalence) is impossible, as follows from
\cite{BiFl,FGN}.

With these examples in mind, following the suggestions in \cite{FlLa1},
we ask: \begin{question*}What conditions on $\Gamma$ allow for strong
deformation retractions of $\X_{\Gamma}(G)$ to $\X_{\Gamma}(K)$ to exist
for all reductive $\C$-groups $G$ with maximal compact subgroup $K\subset G$?
\end{question*}

We call such groups $\Gamma$ \textit{flawed}.\footnote{The name comes from the first two letters of the first named author and the second two letters of the second named author.} When such a deformation retraction does not exist for any non-abelian
$G$, like hyperbolic surface groups, we call the group $\Gamma$
\emph{flawless}.  

\subsection{Summary of Results}
In Section \ref{NC} we prove a general (necessary and sufficient) criterion for flawedness (Theorem \ref{thm:numeric}). Thereafter, we prove our first main theorem (Theorem \ref{thm:main}).  This theorem, although a corollary of the main result in \cite{Berg} and \cite[Corollary 4.10]{FlLa3}  is worth highlighting as it unifies all known cases of flawed groups and brings together the circle of ideas used in \cite{FlLa1, FlLa4,Berg}. 

\begin{alphatheorem}\label{main:thm} 
Let $\Gamma$ be a finite presentable group isomorphic to a free product of nilpotent groups. Then $\Gamma$ is flawed. 
\end{alphatheorem}

In particular, $\PSL(2,\Z)$ is flawed.  Theorem  \ref{main:thm} follows since all finitely generated nilpotent groups are {\it strongly flawed}, a concept we introduce in Section \ref{NC}.

Motivated by work in \cite{PeSo}, in \cite{FlLa4} the authors conjectured that all right angled Artin groups (RAAGs) with torsion are flawed.\footnote{The usual definition of a RAAG does not allow torsion, which is why we say ``with torsion'' to allow for elements to have finite order.  Precisely, these are graph products of cyclic groups.}

As another corollary of Theorem \ref{main:thm}, we see that free products of cyclic groups (with or without torsion) are flawed, giving further evidence for the conjecture that all RAAGs with torsion are flawed.

We then turn our attention to more general RAAGs and prove that in some cases they
are flawed. More precisely, our second main theorem (Theorems \ref{thm:star-RAAG}
and \ref{thm:connectedRAAG}) is this:

\begin{alphatheorem}\label{main:thm2} Let $\Gamma$ be a star-shaped
RAAG $($see Definition \ref{def:star}$)$, then $\Gamma$ is flawed.
Moreover, if $\Gamma$ is a connected RAAG, there is a distinguished
irreducible component $\X_{\Gamma}^{*}(G)\subset\X_{\Gamma}(G)$ such
that, for every reductive $\C$-group $G$ with maximal compact subgroup
$K$, $\X_{\Gamma}^{*}(G)$ strong deformation retracts to $\X_{\Gamma}^{*}(K):=\X_{\Gamma}(K)\cap\X_{\Gamma}^{*}(G)$. 
\end{alphatheorem}

If there is always such a distinguished irreducible component $\X_{\Gamma}^{*}(G)$
that strong deformation retracts to $\X_{\Gamma}^{*}(K)$ as in Theorem
\ref{main:thm2}, then $\Gamma$ will be called \textit{special flawed}
(see Section \ref{boot-section}). 

Also in Section \ref{boot-section},
we define the notion of \textit{$G$-flawed}, allowing for a more
nuanced discussion of flawedness (see Theorems \ref{thm:Ad(G)}, \ref{simpletored}
and \ref{thm:real}).  Precisely, a group $\Gamma$ is $G$-flawed if $\X_{\Gamma}(G)$ strong deformation retracts to $\X_{\Gamma}(K)$ for a fixed $G$.  At the time of this writing, we know of no examples of a group $\Gamma$ that is $G$-flawed but not $H$-flawed for non-abelian non-isomorphic groups $H$ and $G$.

In the penultimate Section \ref{NF}, we prove our third main theorem (Theorems
\ref{thm:nilbyfinite} and \ref{thm:freebyfinite}) which gives classes
of flawed groups that do not fit into either of the aforementioned
classes of flawed groups described in Theorems \ref{main:thm} and
\ref{main:thm2}. Let $\mathsf{F}_{r}$ denote a free group of rank
$r\in\mathbb{N}$.

\begin{alphatheorem}\label{main:thm3} Let $F$ be a finite group,
$\Gamma_{1}\cong F\times\mathsf{F}_{r}$ and $\Gamma_{2}\cong F\times N$
where $N$ is a finitely generated nilpotent group. Then $\Gamma_{1}$
is flawed, and $\Gamma_{2}$ is special flawed. 
\end{alphatheorem}

In the final Section \ref{QC}, we discuss questions and conjectures for further research.

\subsection{Philosophy}
The guiding philosophy of this paper is that quantifying over the geometric structure defining character varieties (the Lie group $G$) gives group-theoretic properties about $\Gamma$, and such persistent structure should be understood as a general feature in (geometric) group theory.  In short, the topology and geometry of the collection of representations of a group $\Gamma$ is an organizing principle for classifying and distinguishing such groups.  

A well-known example of this philosophy is that K\"ahler groups determine uniform singularity types in character varieties \cite{Goldman-Millson}. An explicit (recent) application of the concept of flawed groups discussed in this paper is \cite[Proposition 2.4]{BHJL}.  This latter result uses the notion of flawed groups to prove that if $G$-character varieties of orientable surface groups are isomorphic, then the Euler characteristic of the underlying surfaces are equal.  Flawedness is used to distinguish open surfaces from closed surfaces (which dimension alone cannot do).

\subsection*{Acknowledgments}

Lawton was supported by a Collaboration Grant from the Simons Foundation
(\#578286), and Florentino was supported by CMAFcIO (University of
Lisbon), and the project QuantumG PTDC/MAT-PUR/30234/2017, FCT Portugal.
We thank Henry Wilton and Lior Silberman for helpful comments.  We also thank Maxime Bergeron for suggesting Lemma \ref{berg-lem}.  Lastly, we thank the referee for helpful comments.

\section{A Numerical Criterion for Flawed Groups}

\label{NC}

Let $\Gamma$ be a finitely presentable group. Assume $G$ is a reductive
$\C$-group (a connected reductive affine algebraic group over $\mathbb{C}$).
Then there exists a faithful representation of $G$ and so we may
assume $G\subset\SL(n,\C)$ for some $n$. Moreover, we can arrange
for a maximal compact subgroup in $G$, denoted $K$, such that $K\subset\SU(n)$.
As in the introduction, $\X_{\Gamma}(G)=\hom(\Gamma,G)\quot G$ and
$\X_{\Gamma}(K)=\hom(\Gamma,K)/K$ are the corresponding character
varieties of $\Gamma$. By \cite[Proposition 4.5]{FlLa3}, there is
a natural inclusion of CW complexes: 
\begin{equation}
i_{G}:\X_{\Gamma}(K)\to\X_{\Gamma}(G).\label{eq:iG}
\end{equation}

We adopt the following standard terminology. By a \textit{strong deformation
retraction} (SDR) we mean an inclusion $\iota:Y\hookrightarrow X$
of topological spaces such that the identity map $id_{X}$ is homotopic
(relative to $Y$) to a retract $r:X\to Y$ (that is, $r\circ\iota=id_{Y}$).
We also say ``$Y$ is a strong deformation retract of $X$,'' when
such a SDR $\iota:Y\to X$ exists. 
\begin{defn}
We say that $\Gamma$ is \textbf{flawed} if $i_{G}$ \eqref{eq:iG}
is a SDR for all $G$. We say that $\Gamma$ is \textbf{flawless}
if $i_{G}$ is not a SDR for any non-abelian $G$. 
\end{defn}
Let us list a number of examples which are mostly results from previous
work. 
\begin{example}
All finite groups are flawed. This is an immediate consequence of
the following stronger fact in this simple case.\end{example}
\begin{prop}
If $\Gamma$ is a finite group, then $i_{G}:\X_{\Gamma}(K)\to\X_{\Gamma}(G)$
is a homeomorphism.\end{prop}
\begin{proof}
This is a corollary of many well known facts. For completeness, we
provide an elementary proof. Let $\Gamma$ be finite. Then each $\rho\in\hom(\Gamma,G)$
has finite image and so is polystable and thus $\X_{\Gamma}(G)=\hom(\Gamma,G)/G$.
Also, since the image $\rho(\Gamma)$ is compact it is contained in
a maximal compact subgroup of $G$. Fix a maximal compact $K$ in
$G$. Since all maximal compact subgroups are conjugate, for every
$\rho\in\hom(\Gamma,G)$ there exists $g\in G$ so $g\rho(\Gamma)g^{-1}\subset K$.
Thus, $i_{G}$ is actually surjective, and hence a homeomorphism (since
$\X_{\Gamma}(K)$ is compact).\end{proof}
\begin{example}
\label{exa:Finitely-generated-free}Finitely generated free groups
\cite{FlLa1} are flawed. This was the first example of this general
phenomenon. 
\end{example}

\begin{example}
\label{exa:Finitely-generated-nil}Finitely generated abelian groups
\cite{FlLa4} are flawed. Later, it was shown that finitely generated
nilpotent groups \cite{Berg} are flawed. Also, virtually nilpotent
Kähler groups \cite{BF-nil} are flawed (this includes finite groups). \end{example}
\begin{rem}

As shown in the Appendix of \cite{Berg}, there are examples of finitely
generated nilpotent $\Gamma$ and \textit{non-reductive} $\C$-groups
$G$ so that $\X_{\Gamma}(G)$ is \textit{not} homotopic to $\X_{\Gamma}(K)$
where $K$ is a maximal compact in $G$ (in fact they do not necessarily
have the same number of connected components). So working with the
class of reductive groups $G$ is necessary in the definition of flawed. \end{rem}
\begin{example}
A hyperbolic surface group $\Gamma$ (the fundamental group of a closed
orientable surface $\Sigma$ of genus $g\geq2$) is flawless. This
follows from Theorems 3.15 and 3.12 in \cite{FGN}, together with
the fact that $\X_{\Gamma}(G)$ is homeomorphic to the moduli space
of $G$-Higgs bundles of trivial topological type over a Riemann surface
with underlying topological surface $\Sigma$ (see also Remark \ref{rem:Higgs}
below).\end{example}

\begin{rem}
\label{rem:Higgs}When $\Gamma$ is the fundamental group of a K\"{a}hler
manifold $\Sigma$, $\X_{\Gamma}(G)$ is homeomorphic to $\mathcal{M}_{\Sigma}(G)$,
the moduli space of $G$-Higgs bundles over $\Sigma$ with vanishing
Chern classes; this is one instance of the so-called non-abelian Hodge
correspondence. Likewise, $\X_{\Gamma}(K)$ is homeomorphic to the
moduli space of flat holomorphic principal $G$-bundles over $\Sigma$;
denote it by $\mathcal{N}_{\Sigma}(G)$. In general, $\mathcal{M}_{\Sigma}(G)$
is a \textit{partial} compactification of the cotangent bundle $T^{*}(\mathcal{N}_{\Sigma}(G))$,
which deformation retracts to $\mathcal{N}_{\Sigma}(G)$. However,
one generally expects the boundary divisors in the partial compactification
to change the homotopy type of these moduli spaces. So if this does
not happen, one may be justified in saying that $\Gamma$ has a deficit.
This gives another point-of-view about the name ``flawed.'' 
\end{rem}
Using Kempf-Ness Theory (see \cite{Ne,Sch1,KN}), and Whitehead's
Theorem (\cite[Page 346]{Hatcher}, \cite{Whitehead2}), allows one
to obtain a certain ``numerical'' criterion for flawedness. The
setup is as follows.

Let $G$ be reductive $\C$-group with maximal compact subgroup $K$,
$V$ an affine variety with a rational action of $G$, and $V\quot G:=\mathrm{Spec}_{max}\left(\C[V]^{G}\right)$
the (affine) GIT quotient. By \cite[Lemma 1.1]{Ke}, we may assume
$V$ is equivariantly embedded as a closed subvariety of a (finite
dimensional) $\mathbb{C}$-vector space $\mathbb{V}$, via a representation
$G\to\mathrm{GL}(\mathbb{V})$. 

Let $\langle\ ,\ \rangle$ be a $K$-invariant Hermitian form on $\mathbb{V}$
with norm denoted by $\norm\ \norm$. Define, for any $v\in\mathbb{V}$
the mapping $p_{v}:G\to\mathbb{R}$ by $g\mapsto\norm g\cdot v\norm^{2}$.
It is shown in \cite{KN} that any critical point of $p_{v}$ is a
point where $p_{v}$ attains its minimum value. Moreover, the orbit
$G\cdot v$ is closed and $v\not=0$ if and only if $p_{v}$ attains
a minimum value.

The \textit{Kempf-Ness} set is the set $\mathcal{KN}$ of critical
points $\{v\in V\subset\mathbb{V}\ |\ (dp_{v})_{\id}=0\}$, where
$\id\in G$ is the identity. Since the Hermitian norm is $K$-invariant,
the Kempf-Ness is stable under the action of $K$. The following theorem
is proved in \cite{Sch1} making reference to \cite{Ne}. 
\begin{thm}[Schwarz-Neeman]
\label{schwarzneeeman} The composition $\mathcal{KN}\hookrightarrow V\to V\quot G$
is proper and induces a homeomorphism $\mathcal{KN}/K\to V\quot G$
where $V\quot G$ has the analytic topology. Moreover, $\mathcal{KN}\hookrightarrow V$
is a $K$-equivariant strong deformation retraction.
\end{thm}
In our setting, $V$ is $\hom(\Gamma,G)$ and $G$ acts by conjugation.
Choosing $r$ generators for $\Gamma$, we first embed: 
\[
\hom(\Gamma,G)\subset G^{r}\subset\mathbb{V},
\]
where $\mathbb{V}$ is an affine space where the conjugation action
of $G$ extends, and $\hom(\Gamma,G)\subset\mathbb{V}$ is a closed
$G$-stable subvariety.

Then, following \cite[Proposition 4.7]{CFLO1}, the \emph{Kempf-Ness
set} of $\hom(\Gamma,G)$ is: 
\begin{equation}
\KN_{\Gamma}:=\left\{ (g_{1},\cdots,g_{r})\in\hom(\Gamma,G)\ |\ \sum_{i=1}^{r}[g_{i}^{*},g_{i}]=0\right\} ,\label{eq:KN}
\end{equation}
where $g^{*}$ is the conjugate-transpose of $g$ (defined by a Cartan
involution), and $[g,h]$ denotes $gh-hg$ for $g,h\in G\subset\mathbb{V}$.
It also follows from this definition that $\KN_{\Gamma}$ is $K$-stable
under conjugation, and $\hom(\Gamma,K)\subset\KN_{\Gamma}$. 

We also need the intermediate space $\mathfrak{Y}_{\Gamma}(G):=\hom(\Gamma,G)/K$, which is also a finite CW complex. From the Schwarz-Neeman Theorem
(Theorem \ref{schwarzneeeman}), we conclude: 
\begin{thm}
\label{thm:SN} $\X_{\Gamma}(G)\cong\KN_{\Gamma}/K$ and the natural
inclusion $\KN_{\Gamma}/K\subset\Y_{\Gamma}(G)$ is a SDR. 
\end{thm}
The following result then gives necessary and sufficient conditions
for flawedness. 
\begin{thm}
\label{thm:numeric} Let $\eta:\X_{\Gamma}(K)\to\mathfrak{Y}_{\Gamma}(G)$
be the natural inclusion. Then, the following are equivalent sentences: 
\begin{enumerate}
\item $\Gamma$ is flawed. 
\item $\eta$ induces isomorphisms $\pi_{n}(\X_{\Gamma}(K))\cong\pi_{n}(\mathfrak{Y}_{\Gamma}(G))$
for all $n\in\mathbb{N}$. 
\item The inclusion $\X_{\Gamma}(K)\subset\KN_{\Gamma}/K$ induces isomorphisms
\[
\pi_{n}(\X_{\Gamma}(K))\cong\pi_{n}(\KN_{\Gamma}/K)
\]
for all $n\in\mathbb{N}$. 
\end{enumerate}
\end{thm}
\begin{proof}
This essentially follows from ideas in the proof of \cite[Theorem 4.10]{CFLO1}.

We make two technical notes. First, by \cite[Proposition 4.5]{FlLa3}
the inclusions $i_{G}$, $\eta$, and $\X_{\Gamma}(K)\subset\KN_{\Gamma}/K$
can all be taken to be ``cellular'', that is, there exist CW structures
on these spaces such that the inclusions map onto subcomplexes. Second,
when considering homotopy groups, we always are considering compatible
basepoints with respect to the given inclusions (this is relevant
since in this generality character varieties may not be connected).

We prove (3) is equivalent to (2). We have the commutative diagram
of inclusions: 
\begin{eqnarray*}
\X_{\Gamma}(K) & \stackrel{\eta}{\hookrightarrow} & \mathfrak{Y}_{\Gamma}(G)\\
\varphi\downarrow &  & \ \,\parallel\\
\KN_{\Gamma}/K & \stackrel{\iota}{\hookrightarrow} & \mathfrak{Y}_{\Gamma}(G).
\end{eqnarray*}

By Theorem \ref{thm:SN}, $\iota$ induces isomorphisms on all homotopy
groups. Since $\eta=\iota\circ\varphi$, the induced homomorphisms
in homotopy are $\eta_{*}=\iota_{*}\circ\varphi_{*}$. Thus, $\varphi$
induces isomorphisms on all homotopy groups if and only if $\eta$
does.

Now, we show that (1) and (3) are equivalent. Assume that $\varphi$
induces an isomorphism for all homotopy groups. Then Whitehead's Theorem
(\cite[Page 346]{Hatcher}) implies that $\varphi$ is an SDR since
$\varphi$ maps onto a subcomplex. The conclusion that $\Gamma$ is
flawed then follows from the identification $\KN_{\Gamma}/K\cong\X_{\Gamma}(G)$
in Theorem \ref{thm:SN}.

Conversely, if $\Gamma$ is flawed, then 
\[
\pi_{n}(\KN_{\Gamma}/K)\cong\pi_{n}(\X_{\Gamma}(G))\cong\pi_{n}(\X_{\Gamma}(K)),
\]
for all $n$. \end{proof}
\begin{rem}
For a general group $\Gamma$ (even for a free group) explicitly determining
the Kempf\textendash Ness sets appears to be a very difficult task.
Criterion (2) above avoids the determination of $\KN_{\Gamma}$, relying
only on the topology of the semialgebraic sets $\X_{\Gamma}(K)$ and
$\mathfrak{Y}_{\Gamma}(G)$. However, the corresponding homotopy groups
are also difficult to compute in general. 
\end{rem}
There is a stronger condition than being flawed that turns out to
be sometimes easier to prove in practice. 
\begin{defn}
We will say that $\Gamma$ is \textbf{strongly flawed} if there exists
a $K$-equivariant SDR from $\hom(\Gamma,G)$ to $\hom(\Gamma,K)$.

\end{defn}
\begin{thm}
\label{thm:strongimplies} If $\Gamma$ is strongly flawed, it is
flawed.
\end{thm}
\begin{proof}
If $\Gamma$ is strongly flawed, then there exists a $K$-equivariant
SDR from $\hom(\Gamma,G)$ onto $\hom(\Gamma,K)$ and hence $\Y_{\Gamma}(G)$
strong deformation retracts onto $\X_{\Gamma}(K)$. Therefore, the
inclusion $\eta:\X_{\Gamma}(K)\to\Y_{\Gamma}(G)$ determines isomorphisms
$\pi_{n}(\X_{\Gamma}(K))\cong\pi_{n}(\Y_{\Gamma}(G))$. Thus, by Theorem
\ref{thm:numeric}, $\Gamma$ is flawed. \end{proof}
\begin{rem}
The above theorem, first used in \cite{FlLa1}, has been used to establish the flawedness of just about all groups that are now known to be flawed. Indeed, finitely generated free groups \cite{FlLa1} and finitely generated nilpotent groups \cite{Berg} (Examples \ref{exa:Finitely-generated-free} and \ref{exa:Finitely-generated-nil}) are in fact strongly flawed. The proof in \cite{FlLa4} does not establish the condition of strongly flawed; but from \cite{Berg}, and
as a corollary to Theorem \ref{thm:main}, we now have that finitely
generated nilpotent groups are strongly flawed.
\end{rem}

\begin{rem}
Although we will not use it here, it is natural to say that $\Gamma$ is ``weakly flawed'' if $\pi_{n}(\X_{\Gamma}(G))\cong\pi_{n}(\X_{\Gamma}(K))$
for all $n$.  The fact that flawed implies weakly flawed is obvious since a SDR between spaces implies those spaces are homotopic and hence weakly
homotopic. We know of no examples of weakly flawed groups that are not flawed. 
\end{rem}

\section{Bootstrapping and Extending Flawedness}

\label{boot-section}Consider again the natural inclusion $i_{G}:\X_{\Gamma}(K)\to\X_{\Gamma}(G)$.
One often shows that a group $\Gamma$ is flawed after proving that
$i_{G}$ is a SDR for $G=\SL(n,\C)$, for all $n$, and the general
proof for reductive $G$ is usually simply an adaptation of the $\SL(n,\C)$
case.

\subsection{Bootstrapping Flawedness from the Simple Adjoint Case}

In this subsection, and the next one, we prove that to establish flawedness
it is sufficient, in some cases, to consider only simple groups $G$.
\begin{defn}
We will say that $\Gamma$ is \textit{$G$-flawed} if $i_{G}:\X_{\Gamma}(K)\to\X_{\Gamma}(G)$
is a strong deformation retraction for a fixed $G$ and any maximal
compact $K\subset G$. 
\end{defn}
We start by considering the case when $G$ is a connected reductive
abelian group. It is well known that these are precisely the (affine)
algebraic tori which, over $\mathbb{C}$, are the groups of the form
$T\cong(\mathbb{C}^{*})^{n}$ for some $n\in\mathbb{N}$. 
\begin{prop}
\label{prop:T-flawed}Let $\Gamma$ be a finitely generated group
and $T$ be an algebraic torus. Then $\Gamma$ is $T$-flawed.\end{prop}
\begin{proof}
Let $\Gamma_{Ab}:=\Gamma/[\Gamma,\Gamma]$ be the abelianization of
$\Gamma$, with the canonical epimorphism $\pi:\Gamma\to\Gamma_{Ab}$.
Since $T$ is a commutative group, every representation $\rho:\Gamma\to T$
factors as $\rho=\rho_{Ab}\circ\pi$ for a unique $\rho_{Ab}:\Gamma_{Ab}\to T$.
This shows that the natural inclusion 
\[
\hom(\Gamma_{Ab},T)\subset\hom(\Gamma,T)
\]
is actually an isomorphism of algebraic varieties. Since the conjugation
action by $T$ is trivial, we deduce the isomorphism of character
varieties: 
\[
\X_{\Gamma_{Ab}}(T)\cong\X_{\Gamma}(T).
\]
Finally, since by \cite{FlLa4} every finitely generated abelian group
$\Gamma_{Ab}$ is flawed (and naturally $\X_{\Gamma_{Ab}}(T_{K})\cong\X_{\Gamma}(T_{K})$,
for a maximal compact $T_{K}\subset T$), the same holds for $\Gamma$. \end{proof}
\begin{lem}
\label{lem:product}Let $G$ and $H$ be reductive $\mathbb{C}$-groups.
If $\Gamma$ is strongly $G$-flawed and strongly $H$-flawed, it
is $($strongly$)$ $G\times H$-flawed. \end{lem}
\begin{proof}
Let $K_{G}$ and $K_{H}$ be maximal compact subgroups of $G$ and
$H$, respectively. By assumption, there is a $K_{G}$-equivariant
SDR from $\hom(\Gamma,G)$ onto $\hom(\Gamma,K_{G})$ and a $K_{H}$-equivariant
SDR from $\hom(\Gamma,H)$ onto $\hom(\Gamma,K_{H})$. Since products
of SDRs are SDRs, the natural identification 
\[
\hom(\Gamma,G\times H)=\hom(\Gamma,G)\times\hom(\Gamma,H)
\]
defines a $(K_{G}\times K_{H})$-equivariant SDR from $\hom(\Gamma,G\times H)$
onto $\hom(\Gamma,K_{G}\times K_{H})$, as wanted (and conjugation
by $G\times H$, acts factor-wise). 
\end{proof}
Let $G$ be a reductive $\mathbb{C}$-group, and $F$ be a finite
central subgroup of the maximal compact $K\subset G$. The projection
$\pi:G\to G/F$ induces a morphism of algebraic varieties: 
\[
\hom(\Gamma,G)\to\hom(\Gamma,G/F).
\]
This map is not generally surjective, but it surjects onto a union
of path-connected components of $\hom(\Gamma,G/F)$; see \cite[Lemma 2.2]{G3}
and \cite[Thm. 4.1]{Cul}.
\begin{prop}
\label{prop:lift} If $\hom(\Gamma,K/F)$ is a $K/F$-equivariant
SDR of $\hom(\Gamma,G/F)$, then there exists a $K$-equivariant SDR
from $\hom(\Gamma,G)$ onto $\hom(\Gamma,K)$. \end{prop}
\begin{proof}
Denote by $\hom^{*}(\Gamma,G/F)\subset\hom(\Gamma,G/F)$ the union
of components so that 
\[
\pi:\hom(\Gamma,G)\to\hom^{*}(\Gamma,G/F)
\]
is surjective and consider the commutative diagram: 
\begin{equation}
\xymatrix{
\hom(\Gamma,K) \ar[d]_{\pi_K} \ar@{^{(}->}[r] & \hom(\Gamma,G)\ar[d]_{\pi}\\
\hom^{*}(\Gamma,K/F)  \ar@{^{(}->}[r]& \hom^{*}(\Gamma,G/F),}\label{eq:square}
\end{equation}
where $\hom^{*}(\Gamma,K/F):=\hom(\Gamma,K/F)\cap\hom^{*}(\Gamma,G/F)$,
and $\pi_{K}$ is the restriction of $\pi$ to $\hom(\Gamma,K)$.
It is easy to check that, in fact, $\pi_{K}^{-1}(\hom(\Gamma,K/F))=\hom(\Gamma,K)$.

In \cite[Lemma 3.5]{Lawton-Ramras} it is shown that the varieties
$\hom(\Gamma,G)$ are locally path-connected, so that $\pi$ is a
covering map (by \cite[Lemma 2.2]{G3}), in particular a Serre fibration,
and thus has the homotopy lifting property for maps from arbitrary
CW complexes.

By assumption there is a SDR $\hom(\Gamma,K/F)\hookrightarrow\hom(\Gamma,G/F)$.
This naturally induces a SDR on path-connected components, so on the
bottom of \eqref{eq:square}, we have a strong deformation retraction,
which is a homotopy: $H:I\times\hom^{*}(\Gamma,G/F)\to\hom^{*}(\Gamma,G/F)$
from the identity to a retract $\hom^{*}(\Gamma,G/F)\to\hom^{*}(\Gamma,K/F)$.
By pre-composing with $id_{I}\times\pi$, we get the bottom map in the
following diagram:
\begin{equation}
\xymatrix{  &\hom(\Gamma,G)\ar[d]_{\pi} \\
I\times \hom(\Gamma, G)\ar@{-->}[ur]^{\tilde{H}} \ar[r] & \hom^*(\Gamma,G/F),}\label{eq:homotopy}
\end{equation} which can be lifted to the diagonal arrow, yielding a homotopy $\tilde{H}$ which is a (weak) deformation retraction.

It remains to show that this deformation retraction is an SDR, that is, $\tilde{H}(t,\rho)=\rho$ for all $t$, and all $\rho\in \hom(\Gamma,K)\subset \hom(\Gamma,G)$. By commutativity, we have: 
\[
\tilde{H}(t,\rho)\in\pi^{-1}(H\circ(id_{I}\times\pi)(t,\rho))=\pi^{-1}(H(t,\pi(\rho)))=\pi^{-1}(\pi(\rho)).
\]
Hence, since the fiber is discrete, $\tilde{H}(0,\rho)=\rho$ for all $\rho\in\hom(\Gamma,G)$, and continuity, we conclude that $\tilde{H}(t,\rho)=\rho$ for all $t$ and $\rho\in \hom(\Gamma, K)$, as required. 

Finally, since the multiplication action by $F$ commutes with the
conjugation action by $G$ (because $F$ is central) and since the
conjugation actions of $G/F$ and of $G$ are the same, it is clear
that $\pi$ is $K$-equivariant and restricts to $\pi_{K}$. By assumption,
the bottom homotopy in \eqref{eq:homotopy} is both $K$- and $K/F$-equivariant,
so the lifted homotopy is also $K$-equivariant, by commutativity
of the diagram.
\end{proof}

\begin{rem}
The discreteness of the fiber in the previous argument was essential in showing the SDR lifted to an other SDR.  Alternatively, one could use the fact that the inclusion $\hom(\Gamma, K)\hookrightarrow\hom(\Gamma, G)$ is cofibration since $\hom(\Gamma, K)$ is compact (see \cite{Strom1,Strom2}).
\end{rem}

\begin{thm}
\label{thm:Ad(G)} Let $G$ be a reductive group and $\Ad(G)$ be
its adjoint group. If $\Gamma$ is strongly $\Ad(G)$-flawed, then
it is $($strongly$)$ $G$-flawed. \end{thm}
\begin{proof}
Let $DG=[G,G]$ be the derived subgroup of $G$. Then, $DG$ is semisimple
and there is a central algebraic torus $T$ such that the multiplication
map $T\times DG\to G$ is surjective and has a finite central kernel
$F=T\cap DG$ 
\[
F\to T\times DG\to G\cong T\times_{F}DG
\]
where $F$ acts by identifying $(t,g)\in T\times DG$ with $(tf^{-1},fg)$.
This provides another exact sequence of groups: 
\begin{equation}
F\to G\stackrel{\varphi}{\to}T/F\times DG/F\label{eq:isogeny}
\end{equation}
where $\varphi$ is the homomorphism defined by $\varphi([t,g]):=([t],[g])$
(the notation $[\cdot]$ means the equivalence class under the respective
$F$ action), so that $G/F\cong T/F\times DG/F$.

Since $T/F$ is an abelian reductive group, it is again a torus, so
by Proposition \ref{prop:T-flawed},$\Gamma$ is strongly $(T/F)$-flawed.
Now, since $\Ad(G)$ and $DG$ have the same Lie algebra, we have
an isomorphism: 
\[
\Ad(G)\cong(DG/F)/F_{1}
\]
where $F_{1}$ is the finite center of $DG/F$. By hypothesis, $\Gamma$
is strongly $\Ad(G)$-flawed, and so, by Proposition \ref{prop:lift}
(applied to the $F_{1}$ quotient) it is strongly $(DG/F)$-flawed.
Now, by Lemma \ref{lem:product} $\Gamma$ is $(T/F\times DG/F)$-flawed.
Finally, again by Proposition \ref{prop:lift}, applied to the isogeny
\eqref{eq:isogeny}, $\Gamma$ is $G$-flawed. 
\end{proof}
A reductive $\mathbb{C}$-group is said to be of adjoint type if its
center is trivial. Hence, the above theorem allows us to ``bootstrap''
the property of being flawed from that of being $G$-flawed for all
simple $G$ of adjoint type. 
\begin{cor}
A finitely presented group $\Gamma$ is strongly flawed if and only
if it is strongly $H$-flawed for every simple $\mathbb{C}$-group
$H$ of adjoint type.\end{cor}
\begin{proof}
One direction is trivial. Let $G$ be any reductive $\mathbb{C}$-group.
Then $\Ad(G)$ is semisimple, and it is the direct product of its
simple factors: 
\[
\Ad(G)=G_{1}\times\cdots\times G_{m}.
\]
Since the center of a product is the product of the centers, every
$G_{j}$ is of adjoint type. So, the result follows from Theorem \ref{thm:Ad(G)}
and Proposition \ref{prop:lift}. 
\end{proof}

\subsection{Bootstrapping from the Simple and Simply Connected Case}

We remark that if we assume that $\Gamma$ is $G$-flawed for every
simple and simply connected $G$, this may not be enough to prove
that $\Gamma$ is flawed, since an SDR can be lifted to a covering
but not the other way. However, a weaker general result is still possible,
which motivates the following definition.

\begin{defn}
We say that $\Gamma$ is \emph{special flawed} if there exists an
irreducible component $\X_{\Gamma}^{*}(G)$ in $\X_{\Gamma}(G)$ that
SDR onto $\X_{\Gamma}^{*}(K):=\X_{\Gamma}^{*}(G)\cap\X_{\Gamma}(K)$. \end{defn}
\begin{example}
One trivial example of special flawedness is when a character variety
has an isolated point $[\rho]\in\X_{\Gamma}(G)$; in this case, we
call $\rho$ \emph{rigid}. This situation occurs, for example, when
$\Gamma$ is a Kazhdan group. It is known that the identity representation
of such $\Gamma$ is rigid, for all $G$ (see \cite[Proposition 1]{Rap-PropT}
and \cite[Theorem 3]{Rap}), so Kazhdan groups are special flawed. 
\end{example}
Since SDRs are continuous, we immediately have: 
\begin{prop}
If $\Gamma$ is flawed, it is special flawed. 
\end{prop}
In general, for a group $\Gamma$ and a Lie group $G$ with center
$Z$, $\hom(\Gamma,Z)$ is a group and acts on $\hom(\Gamma,G)$ by
$(\sigma\cdot\rho)(\gamma)=\sigma(\gamma)\rho(\gamma)$. When $\Gamma$
is finitely generated, $\hom(\Gamma,G)$ and $\hom(\Gamma,Z)$ are
naturally subsets of $G^{r}$ with respect to a set of $r$ generators
of $\Gamma$. The action of $\hom(\Gamma,Z)$ is then the restricted
action of $Z^{r}$ on the subspace $\hom(\Gamma,G)\subset G^{r}$
given by $(z_{1},...,z_{r})\cdot(g_{1},...,g_{r})=(z_{1}g_{1},...,z_{r}g_{r})$.
We can also consider the action of $G\times\hom(\Gamma,Z)$ on $\hom(\Gamma,G)$
by combining conjugation and the multiplication action.
\begin{lem}
\label{central-lem} Let $G$ be reductive $\C$-group with maximal
compact subgroup $K$. Let $F$ be a subgroup of $Z(K)$, the center
of $K$. Let $\mathcal{KN}_{1}$ and $\mathcal{KN}_{2}$, respectively,
be the Kempf-Ness sets of $\hom(\Gamma,G)$ with respect to the conjugation
action of $G$, and with respect to the action $G\times\hom(\Gamma,F)$.
Then $\mathcal{KN}_{1}=\mathcal{KN}_{2}$. \end{lem}
\begin{proof}
Without loss of generality, we can assume that $K$ is a subgroup
of $\mathrm{SU}(n)$ and that $G$ is a subgroup of $\mathrm{SL}(n,\C)\subset\mathfrak{gl}(n,\mathbb{C})\cong\C^{n^{2}}$,
and endow $\mathfrak{gl}(n,\mathbb{C})$ with the bilinear form $\langle X,Y\rangle=\tr(X^{*}Y)$.
This form is $K\times Z(K)$-invariant since 
\begin{equation}
\tr((zkXk^{-1})^{*}(zkYk^{-1}))=\tr(kX^{*}k^{-1}z^{-1}zkYk^{-1})=\tr(X^{*}Y),\label{eq:invariance-under-F}
\end{equation}
and the complex conjugate transpose $g^{*}$ of $g\in G$ agrees with
the same operation on the vector space $\mathfrak{gl}(n,\mathbb{C})$.
If $\Gamma$ is generated by $r$ elements, we can consider $\hom(\Gamma,G)\subset G^{r}\subset\C^{rn^{2}}$.
This allows us to define a $K\times\hom(\Gamma,Z(K))$-invariant bilinear
form on $\hom(\Gamma,G)$ by 
\[
\langle(X_{1},...,X_{r}),(Y_{1},...,Y_{r})\rangle=\sum_{j=1}^{r}\langle X_{j},Y_{j}\rangle.
\]
With respect to this form, a representation $\rho\in\hom(\Gamma,G)$
is \textit{\emph{minimal with respect to an action of a group $H$}}
if $\|\rho\|\leq\|h\cdot\rho\|$ for all $h\in H$, and the Kempf-Ness
set with respect to the $H$-action is the collection of minimal vectors. 

Clearly the multiplicative action of $F^{r}$ on $G^{r}$ commutes
with conjugation since $F\subset Z(K)$. Moreover, by \eqref{eq:invariance-under-F}:
\[
\|(g,z_{1},\cdots,z_{r})\cdot\rho\|=\|g\cdot\rho\|,
\]
for all $(g,z_{1},\cdots,z_{r})\in G\times\hom(\Gamma,F)$, using
the conjugation action of $G$ on the right and the conjugation-translation
action of $G\times\hom(\Gamma,F)$ on the left. Thus, the Kempf-Ness
set for both actions are identical.
\end{proof}
Let us denote by $\X_{\Gamma}^{0}(G)$ the path component of $\X_{\Gamma}(G)$
containing the trivial representation $\gamma\mapsto1\in G$, for
all $\gamma\in\Gamma$. Likewise, denote by $\X_{\Gamma}^{0}(K)$
the path component of $\X_{\Gamma}(K)$ containing the trivial representation.

For connected character varieties, we can also bootstrap the property
of being flawed from that of $G$-flawed for all $G$ simple and simply
connected. 
\begin{thm}
\label{simpletored} If $\Gamma$ is strongly $G$-flawed for all
simple simply connected $\C$-groups $G$, then $\Gamma$ is $($strongly$)$
special flawed.\end{thm}
\begin{proof}
Let $G$ be a connected reductive $\C$-group. As in the proof of
Theorem \ref{thm:Ad(G)}, the central isogeny theorem states that
$G\cong T\times_{F_{1}}DG$, with $DG=[G,G]$ semisimple, $T$ a central
algebraic torus, and $F_{1}:=T\cap DG$ a finite central subgroup.

Since $DG$ is semisimple, there exists a collection of simple simply
connected $\C$-groups $G_{1},...,G_{n}$ such that $DG\cong(\prod_{i=1}^{n}G_{i})/F_{2}$,
where $F_{2}$ is finite central subgroup of $\prod_{i=1}^{n}G_{i}$.
Putting this together, there is a finite central group $F:=F_{1}\times F_{2}\subset T\times\prod_{i=1}^{n}G_{i}$
such that $G\cong(T\times\prod_{i=1}^{n}G_{i})/F$.

Let $\hom^{0}(\Gamma,G)$ be the component of $\hom(\Gamma,G)$ that
contains the trivial representation and let $\hom'(\Gamma,F)$ be
the subgroup of $\hom(\Gamma,F)$ mapping $\hom^{0}(\Gamma,T\times\prod_{i=1}^{n}G_{i})$
to itself. Then, by \cite[Proposition 5]{Sikora-SO}, we have: 
\begin{eqnarray*}
\hom^{0}(\Gamma,G) & \cong & \hom^{0}(\Gamma,T\times\prod_{i=1}^{n}G_{i})/\hom'(\Gamma,F)\\
 & \cong & \left(\hom^{0}(\Gamma,T)\times\prod_{i=1}^{n}\hom^{0}(\Gamma,G_{i})\right)/\hom'(\Gamma,F).
\end{eqnarray*}

By Lemma \ref{central-lem}, the Kempf-Ness sets for the conjugation
action of $G$ and that of the mixed action $G\times\hom'(\Gamma,F)$
are the same.

Let $K_{i}$ be maximal compact subgroups of $G_{i}$, and $T_{\R}$
a maximal compact in $T$. By assumption, for each index $i$, there
is a $K_{i}$-equivariant SDR the Kempf-Ness set of $\hom^{0}(\Gamma,G_{i})$
onto $\hom^{0}(\Gamma,K_{i})$. Thus, since products of strong deformation
retracts are strong deformation retracts, and using Proposition \ref{prop:T-flawed}
for the torus case, there is a $(T_{\R}\times\prod_{i=1}^{n}K_{i})$-equivariant
SDR from the Kempf-Ness set of 
\[
\hom^{0}(\Gamma,T)\times\prod_{i=1}^{n}\hom^{0}(\Gamma,G_{i})
\]
to $\hom^{0}(\Gamma,T_{\R})\times\prod_{i=1}^{n}\hom^{0}(\Gamma,K_{i})$.

Note that the conjugation action of $(T_{\R}\times\prod_{i=1}^{n}K_{i})$
and that of $K$ are the same since the conjugation action of $F$
is trivial (it is central). Therefore, by Lemma \ref{central-lem},
there is a $K$-equivariant SDR from the Kempf-Ness set of $\hom^{0}(\Gamma,G)$
onto $\hom(\Gamma,K)$. \end{proof}
\begin{rem}\label{rem-bootsc}
By \cite[Proposition 4.2]{Lawton-Ramras} the above theorem can be
improved in some cases. In particular, when $\Gamma$ is ``exponent-canceling''
(e.g. free groups, free abelian groups, surface groups, RAAGs) the
covering $T\times\prod_{i=1}^{n}G_{i}\to G$ with deck group $F$
induces a surjective map $\X_{\Gamma}(T\times\prod_{i=1}^{n}G_{i})\to\X_{\Gamma}(G)$
whose quotient by $\hom(\Gamma,F)$ induces an isomorphism. In these
cases, the conclusion of Theorem \ref{simpletored} can be improved
from special flawed to flawed. 
\end{rem}

\subsection{Extending Flawedness to the Real Case \label{sub:real-case}}

We now extend the theory to allow for real groups $G$ and real character
varieties. We obtain, in particular, a very general criterion for
when a flawed group $\Gamma$ will be also flawed in the more general
real context.

We call a Lie group $G$ \emph{real $K$-reductive} if the following
conditions hold: 
\begin{enumerate}
\item $K$ is a maximal compact subgroup of $G$; 
\item $G$ is a real algebraic subgroup of $\mathbf{G}(\R)$; the $\R$-points
of a reductive $\C$-group $\mathbf{G}$; 
\item $G$ is Zariski dense in $\mathbf{G}$. 
\end{enumerate}
Any real $K$-reductive Lie group has a faithful representation since
every reductive $\C$-group does. Therefore, we may consider any such
group as a subgroup of $\SL(n,\C)$ for appropriate $n$.

Given a self-map $\alpha$ of a set $X$, we will use the notation
$\Fix_{\alpha}(X):=\{x\in X\ |\ \alpha(x)=x\}$ to denote the fix-point
set.

Let $\lieg$ denote the Lie algebra of $G$, and $\liegc$ the Lie
algebra of $\mathbf{G}$. We will fix a Cartan involution $\theta:\liegc\to\liegc$
which restricts to a Cartan involution 
\begin{equation}
\theta:\lieg\to\lieg.\label{eq:Cartan inv}
\end{equation}

The map $\theta$ is defined as $\theta:=\sigma\tau$, where $\sigma,\tau$
are commuting involutions of $\liegc$, such that $\lieg=\Fix_{\sigma}(\liegc)$
and $\liek':=\Fix_{\tau}(\liegc)$ is the compact real form of $\liegc$.
Thus, $\liek'$ is the Lie algebra of a maximal compact subgroup of
$\mathbf{G}$. We call $\sigma$ a \textit{real structure}.

Using $\theta$ we obtain a Cartan decomposition of $\lieg$: 
\[
\lieg=\liek\oplus\liep
\]
where 
\[
\liek=\lieg\cap\liek',\hspace{0.5cm}\liep=\lieg\cap i\liek'
\]
and $\theta|_{\liek}=1$ and $\theta|_{\liep}=-1$. Furthermore, $\liek$
is the Lie algebra of a maximal compact subgroup $K$ of $G$. Then
$K=K'\cap G$, where $K'$ is a maximal compact subgroup of $\mathbf{G}$,
with Lie algebra $\liek'=\liek\oplus i\liep$. Also, $\liek$ and
$\liep$ satisfy $[\liek,\liep]\subset\liep$ and $[\liep,\liep]\subset\liek$.
The Cartan decomposition of $\liegc$ is: 
\[
\liegc=\liekc\oplus\liepc
\]
with $\theta|_{\liekc}=1$ and $\theta|_{\liepc}=-1$.

The Cartan involution \eqref{eq:Cartan inv} lifts to a Lie group
involution $\Theta:G\to G$ whose differential is $\theta$ and such
that $K=\Fix_{\Theta}(G).$ $\Theta$ is also the composition of two
commuting involutions $T$ and $S$, where $T$ corresponds to $\tau$,
and $S$ corresponds to $\sigma$.\footnote{For an appropriate linear representation of $\mathbf{G}$, we can
arrange for $S$ to be complex conjugation and for $T$ to be complex
conjugation composed with inverse-transpose.}

For a finitely generated group $\Gamma$, there is an inclusion $\hom(\Gamma,K)\hookrightarrow\hom(\Gamma,G)$,
and so there is a natural map $i_{G}:\X_{\Gamma}(K)\to\X_{\Gamma}(G)$.
The map $i_{G}$ is injective by observing that \cite[Remark 4.7]{FlLa3}
applies to this setting by \cite[Section 3.2]{CFLO1}.\footnote{In further generality, a theorem of Cartan says that any connected
real Lie group $G$ is diffeomorphic to $K\times\R^{n}$ where $K$
is a maximal compact subgroup of $G$. The argument in \cite[Remark 4.7]{FlLa3}
can be adapted to this setting too.} 
\begin{defn}
We will say that $\Gamma$ is \textit{$G$-flawed} if $\X_{\Gamma}(G)$
strong deformation retracts onto $i_{G}(\X_{\Gamma}(K))$ for a fixed
$G$ and any maximal compact $K\subset G$. 
\end{defn}
In every known case of a flawed group, it was first showed that $\Gamma$
is $G$-flawed for all reductive $\C$-groups and then that $\Gamma$
was $G$-flawed for all \textit{real} reductive $G$. 
\begin{defn}
If $\Gamma$ is $G$-flawed for all \textit{real} reductive $G$,
we will say that $\Gamma$ is \textit{real flawed}. 
\end{defn}
We conjecture that this is a general phenomenon. In particular, we
conjecture that if $\Gamma$ is flawed, then it is real flawed.

The following theorem, giving evidence to the aforementioned conjecture,
\textit{does} account for all known cases where flawed groups turn
out to be real flawed. 
\begin{thm}
\label{thm:real} If $\Gamma$ is real flawed, it is flawed. Conversely,
if $\Gamma$ is strongly flawed and the SDR commutes with a real structure
on $G$, then $\Gamma$ is real flawed. \end{thm}
\begin{proof}
Since every reductive $\C$-group is a real reductive $K$-group,
$\Gamma$ is flawed if it is real flawed by definition.

Now suppose that $\Gamma$ is strongly flawed. Let $G$ be a real
reductive $K$-group that is a subgroup of the real points of $\mathbf{G}$.
Let $\Theta=ST$ be the Cartan involution on $\mathbf{G}$ such that
$G=\Fix_{S}(\mathbf{G})$, $K=\Fix_{\Theta}(\mathbf{G})$, and $K'=\Fix_{T}(\mathbf{G})$.
Note that $K'$ is a maximal compact subgroup of $\mathbf{G}$ such
that $K=K'\cap G$.

The action of $S$, $T$, and $\Theta$ on $\mathbf{G}$ extends to
an action on $\hom(\Gamma,\mathbf{G})$ by post-composition of homomorphisms.

Since $\Gamma$ is strongly flawed, there exists a $K'$-equivariant
SDR, 
\[
\Phi:\hom(\Gamma,\mathbf{G})\times[0,1]\to\hom(\Gamma,\mathbf{G}),
\]
from $\hom(\Gamma,\mathbf{G})$ onto $\hom(\Gamma,K')$. By assumption,
$\Phi$ is $S$-equivariant. In other words, $S(\Phi(\rho,t))=\Phi(S(\rho),t)$.
Therefore, $\Phi$ restricts to the fix-point set of $S$. Since $\Phi$
is $K'$-equivariant and $K\subset K'$, we conclude that $\Phi$
is $K$-equivariant too. Therefore, since $\Fix_{S}(\hom(\Gamma,\mathbf{G})=\hom(\Gamma,G)$
and $\Fix_{S}(K')=K$, $\Phi$ restricts to a $K$-equivariant SDR
from $\hom(\Gamma,G)$ onto $\hom(\Gamma,K)$. Therefore, $\Gamma$
is (strongly) real flawed. 
\end{proof}

\section{Free Products of Nilpotent Groups are Flawed}

In this section, we prove our first main theorem, as discussed in
the introduction (Theorem A). We start with a generalization of Theorem
\ref{thm:strongimplies} to free products, that may have independent
interest. 
\begin{thm}
\label{thm:freeproduct} A free product of strongly flawed groups
is strongly flawed $($hence flawed$)$. More concretely, let $\Gamma_{1},...,\Gamma_{m}$
be finitely generated groups and $\Gamma_{1}*\cdots*\Gamma_{m}$ be
their free product. If there is a $K$-equivariant SDR from $\hom(\Gamma_{i},G)$
to $\hom(\Gamma_{i},K)$ for all $1\leq i\leq m$, then $\Gamma_{1}*\cdots*\Gamma_{m}$
is strongly flawed. \end{thm}
\begin{proof}
We prove in \cite[Corollary 4.10]{FlLa3} that the hypothesis of this
theorem implies that $\Gamma_{1}*\cdots*\Gamma_{m}$ is flawed. Here,
we provide a short proof of the stronger result. Given $K$-equivariant
strong deformation retracts from $\hom(\Gamma_{i},G)$ to $\hom(\Gamma_{i},K)$
for each $i$, we immediately conclude that there is a $K^{m}$-equivariant
SDR from $\hom(\Gamma,G)\cong\prod_{i=1}^{m}\hom(\Gamma_{i},G)$ to
$\hom(\Gamma,K)\cong\prod_{i=1}^{m}\hom(\Gamma_{i},K)$ since the
product of equivariant SDR's is an SDR that is equivariant with respect
to the product of acting groups. Since the action by $K$ is contained
diagonally in the product action of $K^{m}$, we conclude there is
a $K$-equivariant SDR from $\hom(\Gamma,G)$ onto $\hom(\Gamma,K)$,
as required. \end{proof}

\begin{rem}
\label{freeprod-rem} The proof of Theorem \ref{thm:freeproduct}, can
be directly adapted to the real reductive case as done in \cite{CFLO1,CFLO2}.  In fact, this follows from Theorem \ref{thm:real} with the observation that the SDR commutes with a real structure on $G$.
\end{rem}

\begin{cor}
The class of strongly flawed groups is closed under free product.
\end{cor}

Now, we come to the main result on free products of nilpotent groups.
The lower central series of a group $\Gamma$ is defined inductively
by $\Gamma_{1}:=\Gamma$, and $\Gamma_{i+1}:=[\Gamma,\Gamma_{i}]$
for $i>1$. A group $\Gamma$ is \textit{nilpotent} if the lower central
series terminates to the trivial group. 
\begin{example}
The Heisenberg group 
\[
H(\Z):=\left\{ \left(\begin{array}{ccc}
1 & x & z\\
0 & 1 & y\\
0 & 0 & 1
\end{array}\right)\ |\ x,y,z\in\Z\right\} 
\]
admits the presentation: 
\[
\langle a,b,c\ |\ [a,c]=[b,c]=1,\ [a,b]=c\rangle.
\]
Hence, it is a nilpotent group. 
\end{example}

\begin{thm}\label{thm:main} 
Let $\Gamma$ be isomorphic to a free product of finitely many nilpotent groups, each of which is finitely generated.  Then $\Gamma$ is strongly real flawed. In particular,
if $G$ is be a real reductive $K$-group, then $\X_{\Gamma}(G)$
strong deformation retracts onto $\X_{\Gamma}(K)$. 
\end{thm}

\begin{proof}
As with earlier arguments, we assume that $K$ is a subgroup of $\mathrm{U}(n)$,
$G$ is a subgroup of $\mathrm{GL}(n,\C)$, $g^{*}$ is the complex
conjugate transpose of $g\in G$, and $\rho\in\hom(\Gamma,G)$ is
represented by an $r$-tuple of elements in $G$ since $\Gamma$ is
finitely generated.

Let $\Gamma=\Gamma_{1}*\cdots*\Gamma_{m}$ with every $\Gamma_{i}$ finitely generated and nilpotent. In Bergeron's paper, \cite[Theorem 1]{Berg} it is shown
that there exists a $K$-equivariant strong deformation retraction
of $\hom(\Gamma_{i},G)$ to $\hom(\Gamma_{i},K)$ when $\Gamma$ is
a finitely generated nilpotent group. The main idea of the proof in
\cite{Berg} is this: 
\begin{enumerate}
\item For a real reductive group $G$ acting on a real algebraic variety
$V$, there is a also a \textit{Kempf-Ness} set $\mathcal{KN}_{V}\subset V$
such that there exists a $K$-equivariant strong deformation retraction
from $V$ to $\mathcal{KN}_{V}$; see \cite{RS}. This idea was first
used to study the topology of representation spaces in \cite{FlLa1}
and its generalizations \cite{CFLO1,CFLO2}. 
\item The important observation in \cite{Berg} is that for $\Gamma$ a
nilpotent group one can take $\mathcal{KN}_{\hom(\Gamma,G)}$ to be
$\mathcal{N}_{\Gamma}:=$ 
\[
\{\rho\in\hom(\Gamma,G)\ |\ \rho(\Gamma)\text{ consists of normal elements}\},
\]
where $g\in G$ is normal if and only if $gg^{*}=g^{*}g$. In \cite[Proposition 4.7/4.8]{CFLO1}
the Kempf-Ness set is described generally for $\hom(\Gamma,G)$ where
$\Gamma$ is any finitely generated group. From this description it
is clear that $\mathcal{N}_{\Gamma}\subset\mathcal{KN}_{\hom(\Gamma,G)}$
for any finitely generated $\Gamma$. However, it takes more work
to show these sets are equal if $\Gamma$ is nilpotent. 
\item Therefore, one has a $K$-equivariant strong deformation retraction
from $\hom(\Gamma,G)$ to the set $\mathcal{N}_{\Gamma}$. 
\item Lastly, following \cite{PeSo}, as done in \cite[Section 4]{FlLa4}
for the case where $\Gamma$ is abelian, one shows there is a $K$-equivariant
strong deformation retraction from $\mathcal{N}_{\Gamma}$ onto $\hom(\Gamma,K)$
by applying the scaling SDR $\mathbb{C}^{*}$ to $S^{1}$ to the eigenvalues
of the components of $\rho$. Although it is fairly clear that this
makes sense when $\Gamma$ is abelian, it takes more work to show
this SDR applies in the nilpotent case. 
\end{enumerate}
Thus, Theorem \ref{thm:main} now follows from Theorem \ref{thm:freeproduct}
and Remark \ref{freeprod-rem}.\end{proof}
\begin{cor}
\label{cor:main}A finitely generated group isomorphic to a free product
of nilpotent groups is real flawed. 
\end{cor}

\begin{cor}
The modular group $\PSL(2,\Z)$ is flawed. \end{cor}
\begin{proof}
$\PSL(2,\Z)$ is isomorphic to the free product of $\mathbb{Z}_{2}$
and $\mathbb{Z}_{3}$. Since the free product of finitely many finite
cyclic groups is an example of a free product of nilpotent groups,
we conclude that the modular group is flawed by the above corollary. \end{proof}
\begin{rem}
Theorem \ref{thm:main} includes finitely generated groups that are
free groups, abelian groups, nilpotent groups, and free products of
cyclic groups. Hence this theorem unifies all prior known cases of
$($non-finite$)$ flawed groups. Also, it gives further evidence
that RAAGs with torsion are flawed since free products of cyclic groups
are RAAGs with torsion but were not before now known to be flawed. 
\end{rem}
There is a class of groups that includes both RAAGs and free products
of nilpotent groups, namely, the \textit{graph product of nilpotent
groups}. Let $Q=(V,E)$ be a finite graph and $\{G_{v}\ |v\in V\}$
a collection of finitely generated nilpotent groups. The graph product
of the $G_{v}$'s with respect to the graph $Q$ is defined as $F/N$
where $F$ is the free product of all the $G_{v}'$s and $N$ is the
normal subgroup generated by subgroups of the form $[G_{u},G_{v}]$
whenever there is an edge joining $u$ and $v$. 
\begin{conjecture}
If $\Gamma$ is a finitely presentable group that is isomorphic to
a graph product of nilpotent groups, then $\Gamma$ is flawed. 
\end{conjecture}
The next example emphasizes that the above conjecture still does not
unify all known cases and conjectures about flawed groups. 
\begin{example}
Let $F$ be a finite group that is not nilpotent. Since $F$ is finite
it is flawed. If it were a free product $A*B$ with both $A,B$ non-trivial
then it would be infinite. Hence, since it is not nilpotent, it is
not in the class of groups isomorphic to a graph product of nilpotent
groups. \end{example}
\begin{rem}
The first theorem in \cite{bestvina2020free} states that if a group
$\Gamma$ acts on a simplicial tree $T$ without inversions and with
trivial edge stabilizers, and is generated by the vertex stabilizers
$\Gamma_{v}$, then there is a subset $\mathcal{O}$ of the vertices
of $T$ intersecting each $\Gamma$-orbit in one vertex such that
$G$ is isomorphic to a free product $*_{v\in\mathcal{O}}G_{v}$.
From this point-of-view, Theorem \ref{thm:main} says that finitely
presented groups acting in this way on trees with nilpotent stabilizers
are flawed. 
\end{rem}

\section{Right Angled Artin Groups}

\label{sec:raag}As in the previous section, let $G$ be a reductive
$\C$-group, and let $K\subset G$ be a maximal compact subgroup.
We continue with the assumption (without loss of generality) that
$K\subset\SU(n)$ and $G\subset\SL(n,\C)$.

A \textit{Right Angled Artin Group} (RAAG) is a finitely presented
group having only commutator relations: $ab=ba$. Associated to any
RAAG $\Gamma$ is a graph $Q$ whose vertices correspond to generators
of $\Gamma$ and whose edges correspond to relations in $\Gamma$.
Conversely, given a finite graph $Q$, there exists a RAAG $\Gamma_{Q}$
whose generators correspond to the vertices of $Q$ and whose commutator
relations correspond to the edges of $Q$. A RAAG \textit{with torsion}
is a finitely presented group in which all relations are either commutators or torsion relators ($a^{n}=1$); these are exactly finite graph products of cyclic groups.

Free products of finitely many cyclic groups (no edges in $Q$) and
finitely generated abelian groups ($Q$ is a complete graph) are both
extremal examples of RAAGs (with torsion). Since both these classes
of groups are flawed, in \cite{FlLa4} we conjectured that RAAGs (with
torsion) are flawed. Theorem \ref{cor:main} gives further evidence
of this conjecture. We now summarize a strategy to prove that RAAGs
with torsion are flawed (in the outline $\Gamma$ is a RAAG with torsion).
Define the \textit{elliptic} elements in $G$ to be 
\begin{equation}
{\displaystyle G_{K}:=\bigcup_{g\in G}gKg^{-1}}\label{eq:elliptic}
\end{equation}
(see \cite{Kos}) and let $G_{ss}$ denote the set of semisimple elements
in $G$.
\begin{enumerate}
\item A weak deformation retraction between a space $X$ and a subspace
$A$ is a continuous family of mappings $F_{t}:X\to X$, $t\in[0,1]$,
such that $F_{0}$ is the identity on $X$, $F_{1}(X)\subset A$,
and $F_{t}(A)\subset A$ for all $t$. Define 
\[
\hom(\Gamma,G_{ss}):=\{\rho\in\hom(\Gamma,G)\ |\ \rho(\gamma_{1}),...,\rho(\gamma_{r})\in G_{ss}\}.
\]
Using ideas from \cite{PeSo}, \cite[Lemma 4.15]{FlLa4} proves there
exists a $G$-equivariant weak deformation retraction from $\hom(\Gamma,G)$
onto $\hom(\Gamma,G_{ss})$ that fixes $K$ during the retraction. 
\item Define 
\[
\hom(\Gamma,G_{K}):=\{\rho\in\hom(\Gamma,G)\ |\ \rho(\gamma_{1}),...,\rho(\gamma_{r})\in G_{K}\}.
\]
Using the previous step \cite[Theorem 4.16]{FlLa4} proves that when
$\Gamma$ is a RAAG with torsion, $\X_{\Gamma}(G)$ strongly deformation
retracts onto $\hom(\Gamma,G_{K})\quot G$ and fixes the subspace
$\X_{\Gamma}(K)$ for all time. 
\item By Theorem \ref{thm:numeric}, it remains to prove that there exists
a $K$-equivariant weak retraction from $\hom(\Gamma,G_{K})$ to $\hom(\Gamma,K)$
when $\Gamma$ is a RAAG (with torsion). This step remains an open
problem. 
\end{enumerate}
The above steps provide proof for: 
\begin{thm}
{\cite[Theorem 4.18]{FlLa4}}\label{thm-abelianpaper} Let $\Gamma$
be a RAAG with torsion, $G$ be a reductive $\C$-group and let $K$
be a maximal compact subgroup of $G$. If there exists a $K$-equivariant
weak retraction from $\hom(\Gamma,G_{K})$ to $\hom(\Gamma,K)$ for
all $G$, then $\Gamma$ is flawed. 
\end{thm}
We now illustrate that this last step holds for RAAGs we call \textit{star
shaped}.

\subsection{Star Shaped RAAGs are Flawed}\label{subsec-star}

As described above, given a graph $Q=(V,E)$ with vertex set $V=\{1,\cdots,r\}$
and edge set $E$ (consisting of cardinality 2 subsets of $V$), we
define the \emph{RAAG of $Q$} as the finitely presented group: 
\[
\Gamma_{Q}:=\langle a_{1},\cdots,a_{r}\,|\,[a_{i},a_{j}]=1\mbox{ iff }\{i,j\}\in E\rangle.
\]

\begin{defn}
\label{def:star} Let $(V,E)$ be a star graph, that is $V=\{0,1,\cdots,r\}$,
and the distinguished vertex $0\in V$ is connected to every other
vertex and there are no further edges (in particular, such a graph
is connected). The RAAG associated to this star graph will be called
the \emph{star shaped RAAG} of rank $r+1$. It has the presentation:
\[
\Gamma_{\star}:=\langle a_{0},\cdots,a_{r}\,|\,[a_{0},a_{i}]=1\ \forall i=1,\cdots,r\rangle.
\]

\end{defn}
As before, fix $K$ and $T$, maximal compact and maximal torus, respectively,
of the reductive $\mathbb{C}$-group $G$. Without loss of generality,
consider the Cartan involution on $\SL(n,\mathbb{C})$, given by inverse
conjugate transpose, and its restriction to $G\subset\SL(n,\mathbb{C})$
so that $K=\Fix_{\Theta}G\subset\SU(n)$. 

The torus $T$ can be decomposed into its compact and positive parts
\[
T=T_{K}A,
\]
where $T_{K}=T\cap K=\Fix_{\Theta}T$ is a maximal torus of $K$,
and $A$ is a ``positive real torus'' (e.g. when $G=\GL(n,\mathbb{C})$,
$T$ the diagonal torus, $A$ consists of diagonal matrices with real
positive entries, written as exponentials). 

Let us recall the KAK decomposition on a reductive group (see \cite{knapp}).
Define $^{*}:G\to G$ by $g^{*}:=\Theta(g)^{-1}$, so that $k^{*}=k^{-1}$
for $k\in K$. 
\begin{prop}
Let $G$ be a reductive $\C$-group. Then, every element $g\in G$
may be written as $g=kah^{*}$ for some $k,h\in K$ and $a\in A$.
Moreover, the restricted exponential $\exp:\mathfrak{a}\to A$ $($where
$\mathfrak{a}$ is the Lie algebra of $A)$ is a diffeomorphism and
the element $a\in A$ is unique up to conjugation by the Weyl group
$W$. 
\end{prop}
We need the following result. 
\begin{prop}
\label{prop:commuting-deformation} If $g\in K$ and it commutes with
$ke^{x}h^{*}\in G$ $($$k,h\in K$$)$ then $g$ commutes with $ke^{tx}h^{*}$,
for every $t\in\mathbb{R}$. 
\end{prop}
For the proof we use two lemmata. 
\begin{lem}
\label{lem:commuting-A} If $k\in K$ commutes with $e^{x}\in A$
then it commutes with $e^{tx}$ for every $t\in\mathbb{R}$.\end{lem}
\begin{proof}
Noting that $k^{*}=k^{-1}$, and that $e^{x}$ (a positive definite
matrix in a given linear representation) has a unique $t$ power,
for $t\in\mathbb{R}$, we see that the commuting hypothesis $ke^{x}k^{*}=e^{kxk^{*}}=e^{x}$
is equivalent to $ke^{tx}k^{*}=e^{tx}$ for every $t\in\mathbb{R}$. \end{proof}
\begin{lem}
\label{lem:k-comm-ex} If $k\in K$, $e^{x}\in A$ and $e^{-x}ke^{x}\in K$,
then $k$ commutes with $e^{x}$. \end{lem}
\begin{proof}
If $e^{-x}ke^{x}\in K$, then $1=(e^{-x}ke^{x})(e^{-x}ke^{x})^{*}=e^{-x}ke^{2x}k^{*}e^{-x}$
which implies $ke^{2x}k=e^{2x}$. So by the previous lemma, this means
that $k$ commutes with $e^{x}$. 
\end{proof}
We now prove Proposition \ref{prop:commuting-deformation}. 
\begin{proof}
We start with $g,h,k\in K$, and assume $gke^{x}h^{*}=ke^{x}h^{*}g.$
On ``passing'' $g$ from left to right, we write: 
\[
gke^{x}h^{*}=kg_{1}e^{x}h^{*}=ke^{x}g_{2}h^{*}=ke^{x}h^{*}g_{3},
\]
with $g_{1}=k^{*}gk\in K$, $g_{2}=e^{-x}g_{1}e^{x}$, and $g_{3}=hg_{2}h^{*}=g$
by the hypothesis that $g$ commutes with $ke^{x}h^{*}$. Then, $g_{3}\in K$
which implies $g_{2}=h^{*}g_{3}h\in K$ . By Lemma \ref{lem:k-comm-ex},
both $g_{2}$ and $g_{1}=e^{x}g_{2}e^{-x}\in K$ implies $g_{2}$
commutes with $e^{x}$. So, $g_{1}=g_{2}$. Then, 
\[
gke^{tx}h^{*}=kg_{1}e^{tx}h^{*}=ke^{tx}g_{1}h^{*}=ke^{tx}h^{*}g_{3}=ke^{tx}h^{*}g,
\]
as wanted, for all $t\in\mathbb{R}$. 
\end{proof}
Now denote by: 
\[
\hom(\Gamma,G_{K})
\]
the representations of $\Gamma$ where the image of the generators
lie in the elliptic elements $G_{K}=\cup_{g\in G}gKg^{-1}\subset G$.
It is clear that $\hom(\Gamma,K)\subset\hom(\Gamma,G_{K})$ and using
the embedding $\X_{\Gamma}(K)\hookrightarrow\X_{\Gamma}(G)$ it is
not hard to see that $\X_{\Gamma}(K)$ maps into $\hom(\Gamma,G_{K})\quot G$.
\begin{thm}
\label{thm:star-RAAG}Let $\Gamma_{\star}$ be a star shaped RAAG.
Then, $\hom(\Gamma_{\star},K)/K$ is a strong deformation retract of $\hom(\Gamma_{\star},G_{K})\quot G$
.\end{thm}
\begin{proof}
Write $\rho\in\hom(\Gamma_{\star},G_{K})$ as 
\[
\rho=(A_{0},A_{1},\cdots,A_{r})\in G_{K}^{r+1},\quad\quad A_{i}:=\rho(a_{i}).
\]
Denote by 
\[
\hom_{0}(\Gamma_{\star},G_{K})\subset\hom(\Gamma_{\star},G_{K})
\]
the subset of representations with $A_{0}\in K$. Since, in every
$G$-orbit there is a $G$-conjugate of $A_{0}$ which is already
in $K$, we find that there is a natural identification between the
orbit spaces: 
\[
\hom(\Gamma_{\star},G_{K})\quot G\cong\hom_{0}(\Gamma_{\star},G_{K})/K.
\]
Now, write the KAK decompositions of the $A_{i}$'s: 
\[
A_{i}=k_{i}e^{x_{i}}h_{i}^{*},\quad i=1,\cdots,r.
\]
Then, letting $A_{i}(t):=k_{i}e^{tx_{i}}h_{i}^{*}$, there is a homotopy:
\begin{eqnarray*}
F:[0,1]\times\hom_{0}(\Gamma_{\star},G_{K}) & \to & \hom_{0}(\Gamma_{\star},G_{K})\\
(t,...,A_{i},...) & \mapsto & (...,A_{i}(t),...),
\end{eqnarray*}
with $A_{0}\in K$ kept fixed for all $t$. This is a homotopy of
$\hom_{0}(\Gamma_{\star},G_{K})$ since Proposition \ref{prop:commuting-deformation}
shows that the commutation relations $[A_{i}(t),A_{0}]=1$ are satisfied
for all $t\in[0,1]$. Note that, even though the KAK decomposition
is not unique, $F$ is well defined, since for every $t$
and $A_{i}\in G$, the element $A_{i}(t)$ is the same regardless
of the initial choices ($x_{i},k_{i},h_{i}$) for $A_{i}=A_{i}(0)$.
It is also easy to see that $F$ is indeed continuous, since any sequence
$A_{i}^{(n)}$ ($n\in\mathbb{N}$) converging to $A_{i}$ will give
a sequence $A_{i}^{(n)}(t)$ converging to $A_{i}(t)$ for all $t$
and $i=1,\cdots,r$. Since this homotopy is clearly $K$-equivariant
(even $K\times K$-equivariant), and $\hom(\Gamma_{\star},K)/K$ is kept fixed,
we have determined a SDR from $\hom_{0}(\Gamma_{\star},G_{K})/K$
to $\hom(\Gamma_{\star},K)/K$, as required. \end{proof}
\begin{cor}
\label{cor:Star-RAAG-flawed}Star shaped RAAGs are flawed.\end{cor}
\begin{proof}
Theorem \ref{thm-abelianpaper} and Theorem \ref{thm:star-RAAG} give
the result. 
\end{proof}

\begin{rem}
(1) Star shaped RAAGs are cartesian products of free groups with $\Z$.
Consequently, they are not in the class of groups isomorphic to free
products of nilpotent groups. \\
(2) The above proof can be easily generalized to allow $a_{0}$ to
have torsion. 
\end{rem}

\subsection{Connected RAAGs are Special Flawed}

As above, $G$ is a reductive group $\C$-group, with a Cartan involution
$\Theta:G\to G$, and $K=\Fix_{\Theta}(G)$ is a maximal compact subgroup.
Fix also a maximal torus $T\subset G$ with maximal compact $T_{K}=\Fix_{\Theta}(T)$.

Again, without loss of generality, we assume $G\subset\SL(n,\mathbb{C})$,
$K=G\cap\SU(n)$, $T=\Delta_{n}\cap G$, where $\Delta_{n}$ is the
diagonal torus of $\SL(n,\mathbb{C})$, and the Cartan involution
is given by $\Theta(g)=(g^{-1})^{*}$, where $^{*}$ is conjugate
transpose. We make the following general definitions. 
\begin{defn}
An element $g\in G$ is called \emph{normal} if $g^{*}g=gg^{*}$.
It is called \emph{semisimple} if it is diagonalizable (there is $h\in G$
such that $hgh^{-1}\in T$), and is called \emph{elliptic} if it is
semisimple and all its eigenvalues have norm 1 (this agrees with \eqref{eq:elliptic}).
Finally, $g$ is called \emph{unitary} if $g\in K$. 
\end{defn}
It is clear that unitary elements are both normal and elliptic. The
fact that the converse is also valid, will be crucial later on. 
\begin{lem}
\label{normlem} If $g\in G$ is normal with non-repeating eigenvalues
and $hg=gh$ for some $h\in G$, then $h$ is normal too. \end{lem}
\begin{proof}
Since $hg=gh$ and $g$ is diagonalizable, there exists $k\in\SU(n)$
so $khk^{-1}=t$ and $kgk^{-1}=d$ where $t$ is upper-triangular
and $d$ is diagonal. Thus, $td=dt$. Note that $(dt)_{ij}=d_{ii}t_{ij}$
and $(td)_{ij}=t_{ij}d_{jj}$ and so we conclude that 
\[
0=(dt-td)_{ij}=t_{ij}(d_{ii}-d_{jj})
\]
which in turn gives that $t_{ij}=0$ if and only if $i\not=j$ since
$d_{ii}=d_{jj}$ if and only if $i=j$. Thus, $t$ is diagonal and
hence normal which implies $h$ is normal since it is unitarily diagonalizable. \end{proof}
\begin{lem}
Let $N,E\subset G$ be the subsets of normal and elliptic elements,
respectively. Then $N\cap E=K$. \end{lem}
\begin{proof}
The inclusion $K\subset N\cap E$ is clear. Conversely, let $g\in N\cap E$.
Since $g$ is normal, it is well known that $g$ is unitarily diagonalizable,
that is there is $k\in K$ such that $kgk^{-1}\in T$. Now, note that
the unitary torus $T_{K}$ consists of the elements of $T$ with eigenvalues
of norm 1. Since $kgk^{-1}$ has the same eigenvalues as $g$, and
$g$ is elliptic, $kgk^{-1}\in T_{K}\subset K$. Therefore, $g\in K$. \end{proof}
\begin{cor}
\label{cor:elliptic-to-unitary}Let $g\in G$ be elliptic and $h\in G$
be normal with non-repeating eigenvalues. If $gh=hg$, then $g\in K$. \end{cor}
\begin{proof}
By Lemma \ref{normlem}, since $gh=hg$ and $h\in N$ has non-repeating
eigenvalues, $g$ is also normal. So, $g$ is normal and elliptic.
Thus, by the previous Lemma, $g\in K$. 
\end{proof}
Now, let $\Gamma$ be a finitely generated group, with a fixed collection
of generators $\{\gamma_{1},\cdots,\gamma_{r}\}$. The evaluation
map gives an embedding: 
\[
\hom(\Gamma,G)\hookrightarrow G^{r}.
\]
Let $\mathcal{KN}_{\Gamma}\subset\hom(\Gamma,G)$ be the Kempf-Ness
set and consider the \emph{normal Kempf-Ness subset}: 
\[
\mathcal{N}_{\Gamma}:=\{\rho(\gamma_{i})\in G\mbox{ is normal }\forall i=1,\cdots,r\}.
\]
By the general Kempf-Ness-Neeman-Schwarz theory described earlier,
we have: 
\[
\mathcal{N}_{\Gamma}\subset\mathcal{KN}_{\Gamma},
\]
and the inclusion is $K$-equivariant. Let us say that the marked
group $(\Gamma,\{\gamma_{1},\cdots,\gamma_{r}\})$ is \emph{normal}
if $\mathcal{N}_{\Gamma}=\mathcal{KN}_{\Gamma}$.

Now, let $\Gamma_{Q}$ be a RAAG. Such a group has a natural marking
coming from a graph $Q=(V,E)$, whose set of vertices is precisely
$V=\{\gamma_{1},\cdots,\gamma_{r}\}$: 
\[
\Gamma_{Q}:=\langle\gamma_{1},\cdots,\gamma_{r}\,\mid\,[\gamma_{i},\gamma_{j}]=1\mbox{ iff }\gamma_{i}\gamma_{j}\mbox{ is an edge of }Q\rangle
\]
We say that $\Gamma_{Q}$ is connected if $Q$ is connected. When
the marking is understood, we just write $\Gamma$ instead of $(\Gamma,\{\gamma_{1},\cdots,\gamma_{r}\})$.

Finally, define the subset of elliptic representations of $(\Gamma,\{\gamma_{1},\cdots,\gamma_{r}\})$:
\[
\hom^{e}(\Gamma,G):=\{\rho(\gamma_{i})\in G\mbox{ is elliptic }\forall i=1,\cdots,r\}.
\]
and the subset: 
\[
\hom^{de}(\Gamma,G):=\{\rho(\gamma_{i})\mbox{ is elliptic and has distinct eigenvalues }\forall i=1,\cdots,r\},
\]
which is $G$-invariant. Note that $\hom^{e}(\Gamma,G)$ is just the
set $\hom(\Gamma,G_{K})$ from earlier. 

\begin{thm}\label{thm:connectedRAAG} 
Connected RAAGs are special flawed. \end{thm}
\begin{proof}
We know that the GIT quotient, can be interpreted as the polystable
quotient, and also as the Kempf-Ness (symplectic) quotient: 
\[
\hom(\Gamma,G)\quot G\cong\hom^{ps}(\Gamma,G)/G\cong\mathcal{KN}_{\Gamma}/K,
\]
and for every $G$-invariant subset $Y\subset\hom^{ps}(\Gamma,G)$
we can define the Kempf-Ness set of $Y$ as 
\[
\mathcal{KN}_{\Gamma}^{Y}:=\mathcal{KN}_{\Gamma}\cap Y=\left\{ (A_{1},\cdots,A_{r})\in Y\,\mid\,\sum_{i=1}^{r}[A_{i}^{*},A_{i}]=0\right\} ,
\]
and we get the identification: 
\begin{equation}
Y\quot G\cong\mathcal{KN}_{\Gamma}^{Y}/K,\label{eq:KN_Y}
\end{equation}
as topological (Hausdorff) spaces. In particular, $\mathcal{KN}_{\Gamma}^{Y}$
is always a closed subset of $Y$.

Now, let $\Gamma=\Gamma_{Q}$ be a connected RAAG. Then, there is
a SDR from $\hom(\Gamma,G)\quot G$ to $\hom^{e}(\Gamma,G)\quot G$
by Theorem \ref{thm-abelianpaper}.

Consider the subset $Y:=\hom^{de}(\Gamma,G)$. Then, 
\[
Y\subset\overline{Y}\subset\hom^{e}(\Gamma,G)
\]
which is dense in $\overline{Y}$, the closure of $Y$ in $\hom^{e}(\Gamma,G)$.
Note that $\overline{Y}\quot G$ contains the identity representation.
In fact, commutation relations do not impose any restriction on eigenvalues,
so the distinct eigenvalues condition is the complement of equality
conditions on representations, which form a Zariski closed set in
every irreducible component of $\hom(\Gamma,G)$ \footnote{Note that some components may not intersect with $Y$ at all, see
the next section for examples.}.

Now, we have from Equation \eqref{eq:KN_Y}: 
\[
\mathcal{KN}_{\Gamma}^{\overline{Y}}/K\cong\overline{Y}\quot G.
\]
The following lemma finishes the proof by showing $\mathcal{KN}_{\Gamma}^{\overline{Y}}/K$
is $\X_{\Gamma}^{*}(K)$ for some irreducible component $\X_{\Gamma}^{*}(G)$. \end{proof}
\begin{lem}
If $Y:=\hom^{de}(\Gamma,G)\subset\overline{Y}$, then $\mathcal{KN}_{\Gamma}^{Y}=\hom^{de}(\Gamma,K)$,
and the closure of $\mathcal{KN}_{\Gamma}^{Y}$ equals $\mathcal{KN}_{\Gamma}^{\overline{Y}}$.
Moreover, $\mathcal{KN}_{\Gamma}^{\overline{Y}}/K=\X_{\Gamma}^{*}(K)$
for some irreducible component $\X_{\Gamma}^{*}(G)$. \end{lem}
\begin{proof}
We can $G$-conjugate the first element, $\rho(\gamma_{1})$, to be
in $K$. So, there is an identification: 
\[
\hom^{e}(\Gamma,G)\quot G=\hom_{1}(\Gamma,G)/K,
\]
where $\hom_{1}(\Gamma,G)$ are the representations such that $\rho(\gamma_{1})\in K$.
So, we also have: 
\[
\hom^{de}(\Gamma,G)\quot G\cong\hom_{1}^{de}(\Gamma,G)/K.
\]

Now, let $\rho\in\hom_{1}^{de}(\Gamma,G)$, so that $\rho(\gamma_{1})\in K$
with distinct eigenvalues. For any other vertex, say $\gamma_{2}$,
that is adjacent to $\gamma_{1}$ (i.e., $\rho(\gamma_{1})$ and $\rho(\gamma_{2})$
commute) we get that $\rho(\gamma_{2})$ is elliptic and commutes
with the normal (in fact unitary $\rho(\gamma_{1})$). So, by Corollary
\ref{cor:elliptic-to-unitary}, we get that $\rho(\gamma_{2})\in K$
(and has also distinct eigenvalues). Since $\Gamma$ is a connected
RAAG, we proceed in the same way, and get that all $\rho(\gamma_{i})$
are unitary.

This means that the Kempf-Ness set of $\hom^{de}(\Gamma,G)$ consists
of unitary representations; that is, $\mathcal{KN}_{\Gamma}^{Y}=\hom^{de}(\Gamma,K)$.

Finally, since $\hom^{de}(\Gamma,G)\subset\overline{Y}$ is dense
(and $G$-invariant), then $\mathcal{KN}_{\Gamma}^{Y}\subset\mathcal{KN}_{\Gamma}^{\overline{Y}}$
is also dense. And since $\mathcal{KN}_{\Gamma}^{\overline{Y}}$ is
a closed subset of $\hom^{e}(\Gamma,G)$, we get that the closure
of $\mathcal{KN}_{\Gamma}^{Y}$ equals $\mathcal{KN}_{\Gamma}^{\overline{Y}}$.
But any sequence of matrices $g_{n}$ that verify $[g_{n}^{*},g_{n}]=0$
will have a limit that is normal. Since we are in $\hom^{e}(\Gamma,G)$
the limit will be unitary, by Corollary \ref{cor:elliptic-to-unitary}.
Hence, $\mathcal{KN}_{\Gamma}^{\overline{Y}}$ consists of unitary
representations and includes the identity representation. Thus, $\mathcal{KN}_{\Gamma}^{\overline{Y}}/K$
is $\X_{\Gamma}^{*}(K)$ where $\X_{\Gamma}^{*}(G)$ is the component
of abelian representations (henceforth called the {\it abelian component}). 
\end{proof}

As in Subsection \ref{subsec-star}, let $Q=(V,E)$ be a graph and $\Gamma_{Q}$ be the associated RAAG. $\Gamma_{Q}$ will be called a \emph{tree} if $Q$ is a tree, that is, a connected graph without cycles. By removing a \emph{leaf} $v\in V$ (that is, $v$ is a vertex with valence 1) from a tree $Q$, we get another tree with $n-1$ vertices. The following lemma was suggested to us by M. Bergeron. 

\begin{lem}\label{berg-lem}
Let $G$ be a simply-connected reductive $\C$-group and $\Gamma$
a tree RAAG. Then $\hom(\Gamma,G)$ and $\X_{\Gamma}(G)$ are connected.
\end{lem}

\begin{proof}
First, since $G$ is connected, $\X_{\Gamma}(G)$ is connected if
$\hom(\Gamma,G)$ is connected. So we prove the latter. By \cite{PeSo},
there exists a $G$-equivariant weak deformation retraction from $\hom(\Gamma,G)$
onto $\hom(\Gamma,G_{ss})$. So it suffices to show that there is a path in $\hom(\Gamma, G)$ from any point in $\hom(\Gamma,G_{ss})$ to the identity representation.

Let $\Gamma=\Gamma_{Q}$ be a RAAG with generators $\{a_{1},\ldots,a_{r}\}$,
and $Q=(V,E)$ be a tree with vertices labeled by the integers $\{1,\ldots,r\}$.
There is an edge between $i$ and $j$, that is $\{i,j\}\in E$, if and
only if $[a_{i},a_{j}]=1$ is a relation in $\Gamma$. 

Let $\rho\in\hom(\Gamma,G_{ss})$, with $A_{i}:=\rho(a_{i})\in G_{ss}$.
Then, if $\{i,j\}\in E$, $A_{i}$ and $A_{j}$ are in each other's
centralizers, so that $A_{i}\in C_{G}(A_{j})$ and $A_{j}\in C_{G}(A_{i})$.
Since $G$ is simply-connected, centralizers of semisimple elements
are connected by \cite[Theorem 3.9]{SpSt}. Hence, $C_{G}(A_{i})$
is connected, and contains the identity of $G$, for all $i$.

We can relabel the vertices so that $1\in V$
is a leaf of $Q$, $2\in V$ is a leaf of the tree obtained
by removing $1$, and so on. Let $\sigma(i)\in V$
be the unique vertex connected to $i\in V$ by an edge after removing vertices $1,\ldots,i-1$.  Observe that $\sigma(i)>i$ for all $i=1,\ldots,r$.

Let $\gamma_{1}(t)$, $t\in[0,1]$, be the path in $C_{G}(A_{\sigma(1)})$ from $A_{1}$ to the identity $I\in G$. From this, we construct a 
path $\rho_{1}(t)$ in $\hom(\Gamma,G)$ from $\rho=(A_{1},\ldots,A_{r})=\rho_{1}(0)$
to $\rho_{1}(1)=(I,A_{2},\ldots,A_{r})$. Now, we repeat the process
with a path $\gamma_{2}(t)$ from $A_{2}$ to the identity in $C_{G}(A_{\sigma(2)})$, obtaining a path from $(I,A_{2},\ldots,A_{r})$ to $(I,I,A_{3},\ldots,A_{r})$.
Thus, in a finite number of steps we obtain a path from $\rho$ to $(I,\ldots,I)\in\hom(\Gamma,G)$, since in each step we preserve the relations in $\Gamma$.  This concludes the proof.
\end{proof}

\begin{rem}
Let $\Gamma$ be a tree. From Theorem \ref{thm:connectedRAAG} we know that $\Gamma$ is special flawed.  Therefore, $\Gamma$ is $G$-flawed whenever $\X_\Gamma(G)$ is irreducible.  From Lemma \ref{berg-lem}, we know that $\X_\Gamma(G)$ is connected whenever $G$ is simply-connected.   It is natural to then think, from Theorem \ref{simpletored} and Remark \ref{rem-bootsc}, that trees are $G$-flawed whenever $DG$ is simply-connected.  However, the examples in the next subsection show character varieties of trees $\Gamma$ with simply-connected $G$ which are {\it not} irreducible.  So we cannot conclude that trees are $G$-flawed, even for simply-connected $G$, from Lemma \ref{berg-lem} alone.
\end{rem}

\begin{example}
In \cite{Berg} it is shown that if $\Gamma=\Z\rtimes\Z_{4}$ then
the Kempf-Ness set of $\hom(\Gamma,G)$ is not equal to $\mathcal{N}_{\Gamma}$.
This shows that the strategy used above will not work to prove that
solvable, virtually abelian, virtually nilpotent, nor supersolvable
groups are flawed (although we expect that they are). 
\end{example}

\subsection{Examples: Simple Non-Abelian Non-Free RAAGs}

In this section we consider the simplest non-free, non-abelian RAAG,
which is a star shaped RAAG on 3 vertices. For simplicity, we consider
$G$ to be a simply connected semisimple $\C$-group.

Let $\angle$ be a graph with vertices $\{a,b,c\}$ and edges $\{a,b\}$
and $\{b,c\}$. The corresponding RAAG to $\angle$ admits a presentation:
\[
\Gamma_{\angle}:=\langle a,b,c\ |\ [a,b]=1,\ [b,c]=1\rangle.
\]
For any $\rho\in\hom(\Gamma_{\angle},G)$, letting $B:=\rho(b)$,
we have that $\rho(a):=A$ and $\rho(c):=C$ are in the centralizer
of $B$, denoted $Z_{G}(B)$. Conversely, for any $A,C\in Z_{G}(B)$,
by letting $\rho(a):=A$, $\rho(b):=B$, and $\rho(c)=C$, we define
a $G$-representation of $\Gamma_{\angle}$. Thus,
\begin{equation}
\hom(\Gamma_{\angle},G)=\{(A,B,C)\in G^{3}\ |\ A,C\in Z_{G}(B)\}.\label{eq1}
\end{equation}

Define a map $\pi_{b}:\hom(\Gamma_{\angle},G)\to G$ by $\pi_{b}(\rho)=\rho(b)$. 
\begin{lem}
$\pi_{b}$ is a $G$-equivariant epimorphism. \end{lem}
\begin{proof}
Since $G$ is an algebraic group $\hom(\Gamma_{\angle},G)$ is a subvariety
of $G^{3}$. Thus, $\pi_{b}$ is the restriction (to an algebraic
set) of the projection $G^{3}\to G$, and hence is an algebraic map.
Since for every $B\in G$, there exists $\rho:\Gamma_{\angle}\to G$
defined by $\rho(a)=I=\rho(c)$ and $\rho(b)=B$, we see that $\pi_{b}$
is surjective (and hence an epimorphism). The map $\pi_{b}$ is $G$-equivariant
with respect to conjugation since $\pi_{b}(g\rho g^{-1})=g\rho(b)g^{-1}=g\pi_{b}(\rho)g^{-1}$. 
\end{proof}
By $G$-equivariance we have a map $\hom(\Gamma_{\angle},G)/G\to G/\Ad(G)$,
which restricts to a map $\hom(\Gamma_{\angle},G)^{*}/G\to G/\Ad(G)$.
By post-composing with the projection, $G/\Ad(G)\to G\quot G$ we
obtain a map 
\[
\pi_{b,G}:\X_{\angle}(G)\to G\quot\Ad(G).
\]

\begin{lem}
$\pi_{b,G}$ is continuous and onto, and defines a family over $G\quot\Ad(G)\cong\C^{r}$
where $r=\mathrm{Rank}(G)$. \end{lem}
\begin{proof}
Semisimple elements in $G$ have closed conjugation orbits. Since
for every semisimple $B\in G$, there exists $\rho:\Gamma_{\angle}\to G$
defined by $\rho(a)=I=\rho(c)$ and $\rho(b)=B$, we see that $\pi_{b,G}$
is surjective. Since $\pi_{b}$ is continuous and $G$-equivariant,
the induced map $\hom(\Gamma_{\angle},G)/G\to G/\Ad(G)$ is continuous
with respect to the quotient topology. Consequently, the restriction
of domain to $\hom(\Gamma_{\angle},G)^{*}/G\to G/\Ad(G)$ is continuous.
Finally, the quotient map $G/\Ad(G)\to G\quot G$ is continuous and
so $\pi_{b,G}$ is continuous since composition of continuous maps
is continuous. We note that by \cite{steinberg} that $G\quot\Ad(G)\cong\C^{r}$
where $r$ is the rank of $G$ (since $G$ is simply connected).
\end{proof}
Now define $\hom_{B}(\Gamma_{\angle},G):=\pi_{b}^{-1}(B)$ for $B\in G$.
Thus, we have subvarieties $\pi_{b,G}^{-1}([B]):=\X_{\angle}^{B}(G)$
which are isomorphic to $\hom_{B}(\Gamma_{\angle},G)\quot G$. We
will see examples where these fibers are isomorphic (up to finite
quotients) to free group or free abelian group character varieties.
We handle various cases of the fibers through a series of lemmata.

Equation \eqref{eq1} and the above definitions give: 
\begin{lem}
For every $B\in G$, $\X_{\angle}^{B}(G)\cong(Z_{G}(B)\times Z_{G}(B))\quot G$. 
\end{lem}
Now we consider some special cases. 
\begin{lem}
\label{lem:XB_1}If $B\in Z(G)$, then $\X_{\angle}^{B}(G)$ is an
irreducible component of $\X_{\angle}(G)$. It is isomorphic to $\X_{\mathsf{F}_{2}}(G)$
where $\mathsf{F}_{2}$ is a free group of rank 2. \end{lem}
\begin{proof}
In the special case $B\in Z(G)$, we get $Z_{G}(B)=G$, so the previous
lemma shows that 
\[
\X_{\angle}^{B}(G)\cong G^{2}\quot G\cong\hom(\mathsf{F}_{2},G)\quot G=\X_{\mathsf{F}_{2}}(G),
\]
as wanted.
\end{proof}
The following proposition of Springer and Steinberg will be useful
in our analysis. Recall that an element in $g\in G$ is \textit{regular}
if the dimension of its centralizer $Z_{G}(g)$ is minimal among all
centralizers. This minimal dimension is just the rank of $G$; that
is, the dimension of a maximal torus. In fact, for a sufficient general
semisimple element $g$ (such elements are contained in a maximal
torus $T$) it is true that $Z_{G}(g)=T$. 

The following proposition follows from \cite[Pages 206, 221]{SpSt} and \cite[Rem. 2.10]{steinberg}.

\begin{prop}\label{spst} Let $G$ be a simply connected semisimple $\C$-group. If $B$ is regular, then $Z_{G}(B)$ is a maximal torus.   If $A\in Z_{G}(B)$ and $A,B$ are semisimple, then $A,B$ are contained in the same maximal torus.
\end{prop}

Let $G_{reg}:=\{g\in G\ |\ g\text{ is regular }\}.$
\begin{prop}
\label{prop:XB-reg}Let $G$ be a simply connected semisimple $\C$-group. If $B\in G_{reg}$, then $\X_{\angle}^{B}(G)$
is a subvariety of $\X_{\angle}(G)$ isomorphic to $\X_{\Z^{2}}(G)$.
Moreover, the closure of 
\[
\bigcup_{B\in G_{reg}}\!\!\!\X_{\angle}^{B}(G)
\]
is a subvariety of $\X_{\angle}(G)$ isomorphic to $\X_{\Z^{3}}(G)$.\end{prop}
\begin{proof}
If $B$ is regular then $A,B,C$ commute by Proposition \ref{spst},
so they define an element in $\hom(\Z^{2},T)$, and conversely. Moreover,
if the triple $(A,B,C)$ defines a polystable representation then
it consists of semisimple elements, which suffices to prove the first
statement. We now obtain a family of subvarieties $\hom(\Z^{2},Z_{G}(B))\hookrightarrow\hom(\Gamma_{\angle},G)$
parametrized by $G_{reg}.$ This family consists of triples $(A,B,C)$
of semisimple elements which can be simultaneously conjugated to $T$
and $B$ is regular. The closure, thus consists of such triples that
can be simultaneously conjugated to $T$ (without further restriction
on $B$). \end{proof}
\begin{cor}
If $G=\SL(n,\C)$ or $\mathrm{Sp}(n,\C)$ and $B\in G_{reg}$ and
semisimple, then $\X_{\angle}^{B}(G)\cong\X_{\Z^{2}}(G)$ and is therefore
irreducible.\end{cor}
\begin{proof}
In these cases $Z_{G}(B)$ is a maximal torus $T$. We thus have $\X_{\angle}^{B}(G)\cong\X_{\Z^{2}}(Z_{G}(B))\cong T^{2}/W\cong\X_{\Z^{2}}(G).$
From \cite{FlLa4}, $\hom(\Z^{2},G)\quot G$ is irreducible if and
only if $G$ is simply connected. \end{proof}

\begin{rem}
\label{rem:T/W}(1) More generally, the above corollary holds whenever
$Z_{G}(B)$ is a maximal torus, which is true when $B$ is regular
and $G$ is simple and simply connected, by \cite{steinberg}. (2) In \cite{Sik-Ab,FlLa4} it is shown that, for $G=\SL(n,\C)$ and
$G=\mathrm{Sp}(n,\C)$ there is also an isomorphism $\X_{\Z^{r}}(G)\cong T^{r}/W$,
for every $r\in\mathbb{N}$ where $W$ is the Weyl group of $G$,
acting diagonally. See also \cite{FS}. 
\end{rem}

\begin{prop}\label{prop:notpoly}
Let $G=\SL(n,\C)$, $B\in G$ and $A,C\in Z_G(B)$.  If $B$ is not diagonalizable but is non-derogatory, then the triple $(A,B,C)$ corresponding to a point in $\hom(\Gamma_\angle,G)$ is not polystable.
\end{prop}

\begin{proof}
$B$ is non-derogatory if and only if its minimal polynomial is equal to it characteristic polynomial (this is equivalent to the eigenvalues of the Jordan blocks of $B$ being distinct).  By \cite[Theorem 3.2.4.2]{HoJo} , we conclude that $A$ and $C$ are polynomials in $B$.  Since $B$ is upper-triangulizable it follows that all three are upper-triangularizable and hence $(A,B,C)$ is not polystable since $B$ is non-diagonalizable.
\end{proof}

We now look at a couple of special cases.

\subsubsection{$G=\SL(2,\C)$.}

We consider three cases for $B$: 
\begin{enumerate}
\item $B=\pm I:=\pm\left(\begin{array}{cc}
1 & 0\\
0 & 1
\end{array}\right)$, 
\item $B$ is conjugate to $J_{\pm}:=\left(\begin{array}{cc}
\pm1 & 1\\
0 & \pm1
\end{array}\right)$, or 
\item $B$ is regular (conjugate to a diagonal matrix with distinct eigenvalues). 
\end{enumerate}
In Case (1), by Lemma \eqref{lem:XB_1} $\X_{\angle}^{B}(G)$ is isomorphic
to $\X_{\mathsf{F}_{2}}(G)$ which is in this case $$\SL_{2}(\C)^{2}\quot\SL(2,\C)\cong\C^{3}$$
(see \cite{Vogt}). $\X_{\mathsf{F}_{2}}(\SL(2,\mathbb{C}))$
strong deformation retracts to $\SU(2)^{2}/\SU(2)$ which is homeomorphic
to a closed 3-ball (see \cite{FlLa1}).

In Case (2), $\X_{\angle}^{B}(G)$ is empty. Indeed, by Proposition \ref{prop:notpoly}
all triples $(A,B,C)$ with $B=\pm J$ such that $B$
commutes with both $A$ and $C$ are simultaneously upper-triangular; so the corresponding representation will not be polystable.

In Case (3), since $B$ is regular, by combining \eqref{prop:XB-reg}
with Remark \eqref{rem:T/W}, we get $\X_{\angle}^{B}(G)\cong T^{2}/W$
where $W\cong\mathbb{Z}_{2}$ is the Weyl group corresponding to $T\cong\mathbb{C}^{*}$.
This space strong deformation retracts to 
\[
\hom(\Z^{2},\SU(2))/\SU(2)\cong(S^{1})^{2}/\Z_{2}\cong S^{2};
\]
see \cite[Page 20]{CFLO1}. By Proposition \eqref{prop:XB-reg} we
have: 
\[
\overline{\bigcup_{B\in G_{reg}}\!\!\!\X_{\angle}^{B}(G)}\cong\X_{\Z^{3}}(G)\cong T^{3}/W.
\]
This space strong deformation retracts to 
\[
\hom(\Z^{3},\SU(2))/\SU(2)\cong(S^{1})^{3}/\Z_{2},
\]
which is a 3-dimensional orbifold; see \cite[Page 23]{CFLO1} for
details and a visualization.

Summarizing, for $G=\SL(2,\mathbb{C})$, $\X_{\angle}(G)$ consists
of exactly three irreducible (Zariski closed) components, two corresponding
to Case (1) and isomorphic to a free group character variety $\X_{\mathsf{F}_{2}}(G)\cong\mathbb{C}^{3}$,
and the third, corresponding to Case (3), is isomorphic to a free
abelian group character variety $\X_{\Z^{3}}(G)\cong T^{3}/W$ (also
3 dimensional). See Figure \ref{sl2fig} for a schematic drawing of
this example. We can see explicitly that each component strong deformation
retracts to the corresponding $\SU(2)$-character variety and the
union of those retracts is exactly $\X_{\angle}(\SU(2))$. 

Note that the singular locus of each component $\X_{\angle}^{\pm I}(G)$,
where $A$ and $C$ commute (and $B=\pm I$), is exactly where they
intersect the other irreducible component (see \cite{FlLa2}).

\begin{figure}[ht!]
\includegraphics[scale=0.4]{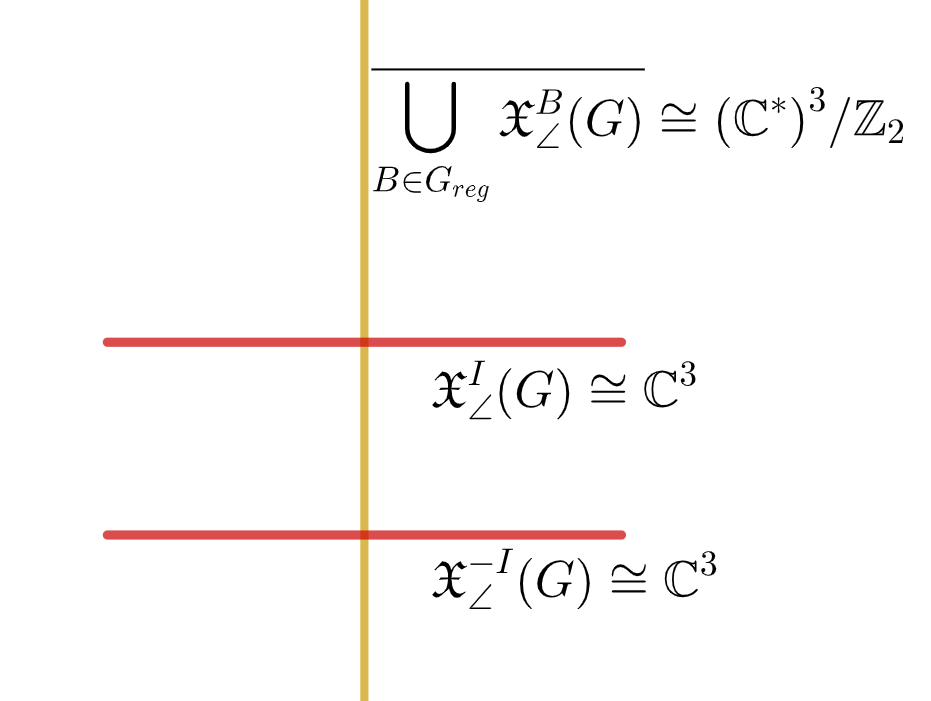} \caption{$\X_{\angle}(\SL(2,\C))$}
\label{sl2fig} 
\end{figure}

\subsubsection{$G=\SL(3,\C)$.}

Now there are four cases to consider for the ``node'' $b$: 
\begin{enumerate}
\item $B$ is central, so that $B\in\{I,\,\omega I,\,\omega^{2}I\}$ where
$I$ is the identity matrix, and $\omega$ is a third root of unity.
By Lemma \eqref{lem:XB_1}, $\X_{\angle}^{B}(G)\cong\X_{\mathsf{F}_{2}}(G),$
where $\mathsf{F}_{2}$ is a free group of rank 2, a branched double cover of $\C^8$ (see \cite{La0,law}) homotopic to $S^8$ (see \cite{FlLa1}).  

\item $B$ is upper triangularizable but not diagonalizable (and hence has
repeated eigenvalues). Writing $B$ in Jordan form, one concludes
that the commutation of $B$ with $A$ and $C$ implies that either $A,C$ are both upper-triangular too (Proposition \ref{prop:notpoly}), or that they are both simultaneously upper-triangulizable with $B$ (an easy calculation).  Either way the triple $(A,B,C)$ does not correspond to a polystable $G$-representation. Hence $\X_{\angle}^{B}(G)$ is empty (by definition). 
\item $B$ is regular (diagonalizable with distinct eigenvalues), which
implies 
\[
\X_{\angle}^{B}(G)\cong\X_{\Z\times\Z}(G)\cong T^{2}/W,
\]
by Proposition \eqref{prop:XB-reg} where $T\cong(\mathbb{C}^{*})^{2}$
is a maximal torus and the Weyl group is the symmetric group on 3
letters $W\cong S_{3}$. So, this is a 4-dimensional orbifold.
\item $B$ is diagonalizable with two repeated eigenvalues, but not central.
This last case is new (compared to the case $\SL(2,\C)$) and so we
detail it below. 
\end{enumerate}
Denote $B_{\lambda}:=\left(\begin{array}{ccc}
\lambda & 0 & 0\\
0 & \lambda & 0\\
0 & 0 & \lambda^{-2}
\end{array}\right)$ where $\lambda^{3}\not=1$. If $A,C\in Z_{G}(B)$ and $(A,B,C)$
is polystable, then $A,C$ are both of the form $\left(\begin{array}{cc}
X & \mathbf{0}^{\dagger}\\
\mathbf{0} & \mathrm{Det}(X)^{-1}
\end{array}\right)$ where $X\in\GL(2,\C)$, $\dagger$ signifies transpose, and $\mathbf{0}=(0,0)$.
Hence, $\X_{\angle}^{B}(G)\cong\X_{F_{2}}(\GL(2,\C))$, which is a
variety of dimension 5.

By varying $B$ over regular elements (a 2 dimensional variety) in
Case (3), and over $B_{\lambda}$ in Case (4), and taking their closures
as before, we obtain 2 irreducible components:
\[
\overline{\bigcup_{B\in G_{reg}}\!\!\!\X_{\angle}^{B}(G)},\quad\mbox{ and }\quad\overline{\bigcup_{B_{\lambda}}\X_{\angle}^{B_{\lambda}}(G)},
\]
both of dimension 6, and both fibered by $\SL(3,\mathbb{C})$-character
varieties of either $\mathbb{Z}^{2}$ or $\mathsf{F}_{2}$. Figure
\ref{sl3fig} is a schematic drawing of this example.  Note that the singular locus of each $\X_\angle^{B_\lambda}(G)$ (the blue curve) intersects the abelian locus $T^3/S_3$ (yellow curve), so the diagram is slightly misleading in that there is a continuum of such intersections.

As with $\SL(2,\C)$, each of these cases corresponds to a character
variety known to strong deformation retract as required, and the SDR restricts to the intersections, providing a SDR on the whole space.

\begin{figure}[ht!]
\includegraphics[scale=1.6]{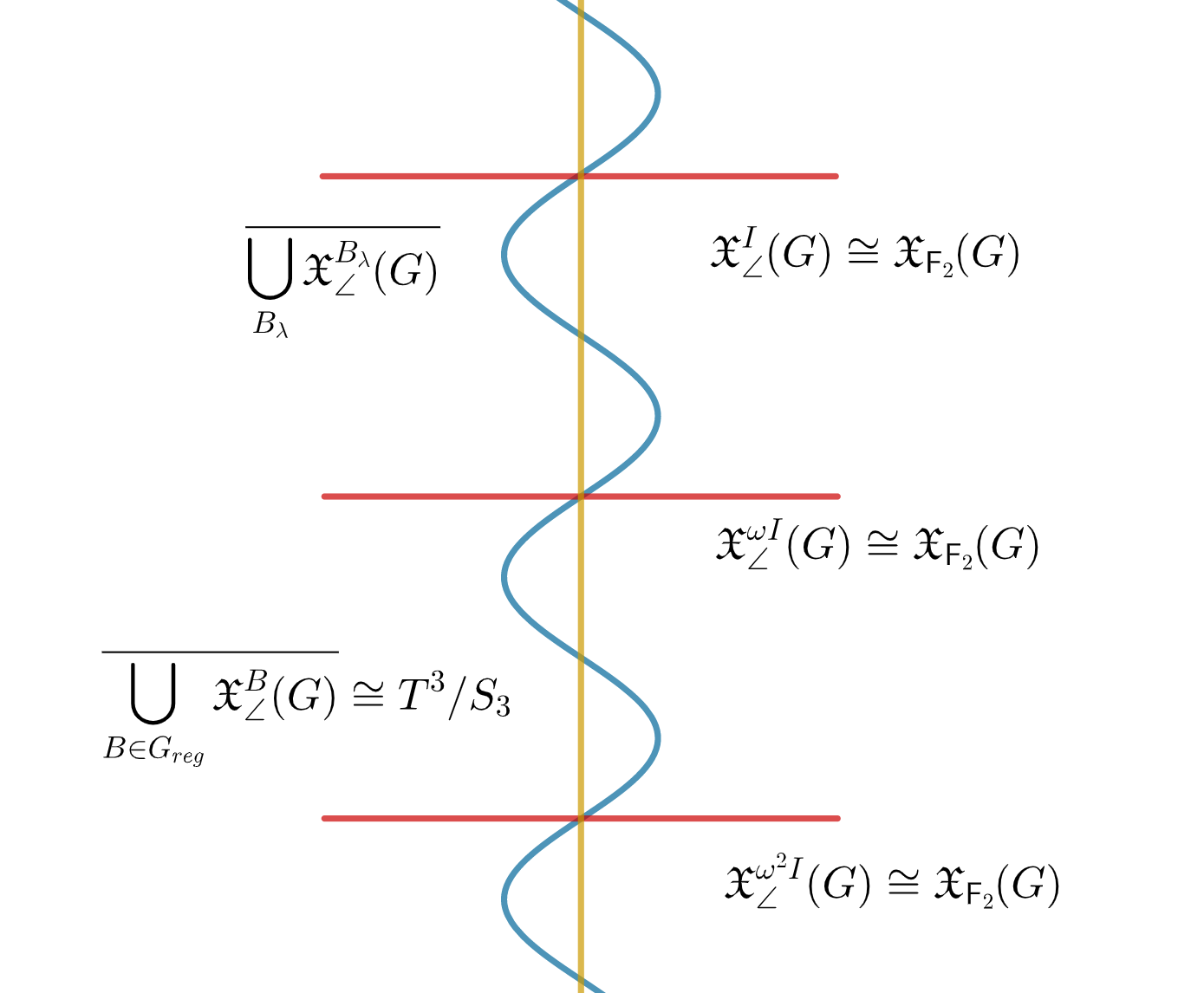} \caption{$\X_{\angle}(\SL(3,\C))$}
\label{sl3fig} 
\end{figure}

Notice that the abelian component in both the above examples intersects every other component.  It would be interesting, given Theorem \ref{thm:connectedRAAG}, to determine if this is a general phenomenon for connected RAAGs.

\section{Direct Products with Finite Groups}

\label{NF}In this last section, we consider some classes of flawed
groups which have a finite group $F$ as a direct (cartesian) factor.
We consider products $\Gamma=F\times G$ where $G$ is either free
or nilpotent. This extends the class of flawed groups further as finite-by-nilpotent
and finite-by-free groups are not, in general, free products of nilpotent
groups as the next example shows. 
\begin{example}
Let $A$ and $B$ be groups of order greater than 3 (possibly infinite).
By \cite[Proposition 4]{trees} the free product $A*B$ contains a
free group of rank at least 2. Let $\Gamma$ be a direct product of
an finitely generated abelian groups $A$ with a finite non-nilpotent
group $F$. Then $\Gamma$ is finite by nilpotent but not nilpotent.
If $\Gamma$ was a free product of two or more non-trivial nilpotent
groups, then it would contain a free group of rank at least 2 by the
above reference; which it does not. 
\end{example}

\subsection{Direct Products of Finite Groups with Free Groups are Flawed}

Let $\Gamma$ be isomorphic to a product $F\times\mathsf{F}_{r}$,
where $F$ is a finite group and $\mathsf{F}_{r}$ is a free group
of rank $r$. 
\begin{thm}
\label{thm:freebyfinite} Let $F$ be a finite group. Then $\Gamma\cong\mathsf{F}_{r}\times F$
is flawed.\end{thm}
\begin{proof}
Consider a presentation of $\Gamma$ of the form: 
\[
\Gamma=\left\langle a_{1},\cdots,a_{r},b_{1}\cdots\,b_{s}\,\mid\,[a_{i},b_{j}]=1,\ R_{k}(b)\right\rangle ,
\]
where $R_{k}(b)$ denote relations only among the $b_{j}$'s (the
$a_{i}$'s are free generators of the $\mathsf{F}_{r}$ factor).

Let $\rho:\Gamma\to G$ be a polystable representation to a reductive
$\mathbb{C}$-group, and denote $A_{i}:=\rho(a_{i})$ and $B_{j}:=\rho(b_{j})$.
First we note that all $B_{j}$'s are elliptic. Indeed, being of finite
order implies that all their eigenvalues are complex numbers of unit
norm. Moreover, since the group they generate is compact, and all
such maximal compact are conjugated, there is a $g\in G$ that simultaneously
conjugates all $B_{j}$ into our fixed maximal compact $K\subset G$.

The proof now proceeds as in the proof of Theorem \ref{thm:star-RAAG}:
Denote by 
\[
\hom_{0}(\Gamma,G)\subset\hom^{ps}(\Gamma,G)
\]
the subset of polystable representations with all $B_{i}\in K$. Since,
in every $G$-orbit there is $g\in G$ so so that $g\cdot B_{j}\in K$,
there is a natural identification between the orbit spaces: 
\[
\hom^{ps}(\Gamma,G)\quot G\cong\hom_{0}(\Gamma,G)/K.
\]
Now, as before, write the KAK decompositions: $A_{i}=k_{i}e^{x_{i}}h_{i}^{*},\quad i=1,\cdots,r$,
let $A_{i}(t):=k_{i}e^{tx_{i}}h_{i}^{*}$, and define a homotopy:
\begin{eqnarray*}
H:[0,1]\times\hom_{0}(\Gamma,G) & \to & \hom_{0}(\Gamma,G)\\
(t,A_{i},B_{j}) & \mapsto & (A_{i}(t),B_{j})
\end{eqnarray*}
(so $B_{j}\in K$ are kept fixed for all $t$). As before, $H$ is
well defined and continuous since for every $t$ and $A_{i}\in G$,
the element $A_{i}(t)$ is the same regardless of the initial choices
($x_{i},k_{i},h_{i}$) for $A_{i}=A_{i}(0)$. Moreover, Proposition
\ref{prop:commuting-deformation} shows that the commutation relations
$[A_{i}(t),B_{j}]=1$ are satisfied for all $t\in[0,1]$. Since this
homotopy is $K$-equivariant, and $\hom(\Gamma,K)/K$ is kept fixed,
we have determined a SDR from $\hom_{0}(\Gamma,G)/K$ to $\hom(\Gamma,K)/K$,
and so a SDR $\hom(\Gamma,K)/K\hookrightarrow\hom^{ps}(\Gamma,G)\quot G$
as wanted. 
\end{proof}

\subsection{Direct Products of Finite Groups with Nilpotent Groups are Special Flawed}
\begin{thm}
\label{thm:nilbyfinite} If $\Gamma$ is a direct product of a nilpotent
group with a finite group, then $\Gamma$ is special flawed.\end{thm}
\begin{proof}
Let $\Gamma=N\times F$, where $N$ is nilpotent and $F$ is finite,
and let $a_{1},\cdots,a_{r}$ be generators of $N$ and $b_{1},\cdots,b_{s}$
be generators of $F$. For a representation $\rho:\Gamma\to G$ write
$A_{i}=\rho(a_{i})$ and $B_{i}=\rho(b_{i})$, so that: 
\[
[A_{i},A_{j}]=[A_{i},B_{j}]=1
\]
Now, let $\rho:\Gamma\to G$ be a polystable representation. Then
the $A_{i}=\rho(a_{i})$ generate a reductive nilpotent group. Hence,
by Bergeron's result \cite{Berg}, there is some $g\in G$ , such
that $g\cdot A=(gA_{1}g^{-1},\cdots,gA_{r}g^{-1})$ is an $r$-tuple
of normal matrices.

Again, let $Y:=\hom^{de}(\Gamma,G)\subset\overline{Y}\subset\hom(\Gamma,G)$
be the subset of representations which have all $A_{i}$ with non-repeating
eigenvalues. Then, the Kempf-Ness set of $Y$ is given by: 
\[
\mathcal{KN}_{\Gamma}^{Y}=\{(A_{1},\cdots,A_{r},B_{1},\cdots,B_{s})\in Y\ \mid\ A_{j}\mbox{ are normal and }B_{j}\mbox{ are unitary}\}.
\]
Indeed, $\mathcal{KN}_{\Gamma}^{Y}\subset\mathcal{KN}_{\Gamma}$ consists
of matrices with minimum Frobenius norm in each $G$-orbit, normal
matrices have the minimum norm in their respective $G$-orbits and,
since $B_{j}$'s are elliptic, by Corollary \ref{cor:elliptic-to-unitary},
they are in fact unitary.

Now, by the same argument as before, we get that $\mathcal{KN}_{\Gamma}^{\overline{Y}}$
is the closure of $\mathcal{KN}_{\Gamma}^{Y}$, which means that:
\[
\mathcal{KN}_{\Gamma}^{\overline{Y}}=\{(A_{1},\cdots,A_{r},B_{1},\cdots,B_{s})\in\overline{Y}\ \mid\ A_{j}\mbox{ are normal and }B_{j}\mbox{ are unitary}\}.
\]
Finally, using the eigenvalue scaling map \cite{FlLa4}, we can show
that there is a $K$-equivariant SDR from $\mathcal{KN}_{\Gamma}^{\overline{Y}}$
to $\overline{\hom^{de}(\Gamma,K)}$, inducing an SDR from $\mathcal{KN}_{\Gamma}^{\overline{Y}}/K$
to $\overline{\hom^{de}(\Gamma,K)}/K$ as wanted. 
\end{proof}

\section{Questions and Conjectures} \label{QC} 

In this final section we list some questions and conjectures for further research:

\begin{enumerate}
\item From the work in \cite{munozsl2,munozsu2} it is clear that torus knots are $\SL(2,\C)$-flawed. We will call a group a \textit{generalized torus link group} if it
can be presented as $$\langle a_{1},...,a_{r}\ |\ a_{i}^{n_{i}}=a_{j}^{n_{j}}\text{ for all }i,j\rangle$$ for positive integers $n_{1},...,n_{r}$. When $r=2$ these are torus link groups and if further $\gcd(n_1,n_2)=1$
these are torus knot groups. We conjecture that generalized
torus link groups are flawed.
\item We know that closed hyperbolic surface groups are flawless and that was shown using
Higgs bundle theory. Given the work in \cite{FGN}, we conjecture that all non-abelian K\"ahler groups
are flawless.
\item A group is {\it supersolvable} if it admits an invariant normal series where all the factors are cyclic groups.  Finitely generated supersolvable groups generalize finitely generated nilpotent groups. Given results in \cite{Borel-Serre,SpSt}, we conjecture that finitely generated supersolvable groups are flawed. 
\item We have shown free products of nilpotent groups are flawed (Theorem \ref{main:thm}). Are free
products of nilpotent groups amalgamated over abelian groups also
flawed?
\item Thinking more like a geometric group theorist, if two groups are commensurable,
and one is flawed, is the other also flawed?  
\end{enumerate}
\newcommand{\etalchar}[1]{$^{#1}$}
\def\cdprime{$''$} \def\Dbar{\leavevmode\lower.6ex\hbox to 0pt{\hskip-.23ex
  \accent"16\hss}D}
\providecommand{\bysame}{\leavevmode\hbox to3em{\hrulefill}\thinspace}
\providecommand{\MR}{\relax\ifhmode\unskip\space\fi MR }
\providecommand{\MRhref}[2]{%
  \href{http://www.ams.org/mathscinet-getitem?mr=#1}{#2}
}
\providecommand{\href}[2]{#2}

\end{document}